\documentclass[a4paper]{amsart}

\usepackage{amssymb}
\usepackage{amsthm}  
\usepackage{amsmath} 
\usepackage{amscd} 
\usepackage[all]{xypic}
\usepackage{url}
\usepackage{multirow}
\usepackage{hyperref}
\hypersetup{linktocpage}

\newcommand{\om}{\omega}

\newcommand{\si}{\sigma}

\newcommand{\ba}{\mathcal{G}}

\newcommand{\fg}{\frak g}
\newcommand{\fp}{\frak p}

\newcommand{\fl}{\frak l}
\newcommand{\fh}{\frak h}

\newcommand{\Ad}{{\rm Ad}}

\newcommand{\id}{{\rm id}}

\newcommand{\U}{\Upsilon}

\newcommand{\Rho}{{\mbox{\sf P}}}

\newcommand{\ad}{{\rm ad}}

\newtheorem{thm}{Theorem}[section]
\newtheorem{prop}[thm]{Proposition}
\newtheorem{lem}[thm]{Lemma}
\newtheorem{cor}[thm]{Corollary}

\theoremstyle{definition}
\newtheorem{def*}[thm]{Definition}

\begin{document}
\title{Almost conformally almost Fedosov structures}
\author{Jan Gregorovi\v c}
\address{Department of Mathematics and Statistics, Faculty of Sciences \\
Masaryk University \\
Kotl\' a\v rsk\' a 2 \\ 
Brno, 611 37 \\ 
Czech Republic and Faculty of Mathematics, University of Vienna, Oskar Morgenstern Platz 1, 1090 Wien, Austria}
\email{jan.gregorovic@seznam.cz}
\thanks{The author was partially supported by the project P29468 of the Austrian Science Fund (FWF). The author would like to thank L. Zalabova for reading and commenting the earlier versions of the article.}

\begin{abstract} 
We study the relations between the projective and the almost conformally symplectic structures on a smooth even dimensional manifold. We describe these relations by a single almost conformally symplectic connection with totally trace--free torsion sharing the geodesics (up to parametrization) with the projective class. This connection generalizes a (conformally) Fedosov structure depending on the remaining torsion of this distinguished connection. In fact, we see these structures as the almost symplectic analogy of the conformal Riemannian structures, because there is an analogy of the class of Weyl connections on the conformal Riemannian structure. Moreover, such a class encodes the variability of the connections in a projective class. The distinguished connection trivializes such a class of Weyl connection. There is a description these geometric structures as Cartan geometries that generalizes the description of projective and conformal Riemannian structures as parabolic geometries. This makes possible to construct an analogy of the so--called Bernstein-Gelfand-Gelfand sequences and Bernstein-Gelfand-Gelfand complexes.
\end{abstract}
\maketitle
\tableofcontents

\section{Introduction}

Let $M$ be a smooth connected manifold of even dimension $n\geq 6$. \emph{An almost conformally symplectic structure (shortly ACS--structure)} on $M$ is a smooth line subbundle $\ell\subset \wedge^2 T^*M$ such that for each $x\in M$, each non--zero element of $\ell_x$ is a non--degenerate bilinear form on $T_xM$. The ACS--structures are almost symplectic analogies of the well--known conformal Riemannian structures, which are the smooth line subbundles of $S^2 T^*M$ satisfying the same non--degeneracy condition. However, there is a large difference between the conformal Riemannian structures and the ACS--structures. The conformal Riemannian structures are geometric structures of second order, while the ACS--structures are geometric structures of infinite order, cf. \cite{Kob}. Moreover, the torsion--free connections preserving the conformal Riemannian structure form the class of so--called Weyl connections depending on a one--form, cf. \cite[Section 1.6]{parabook}, while the connections preserving the ACS--structure (shortly \emph{ACS--connections}) with the same (in general non--zero) torsion form a class depending on a symmetric trilinear form. The class of Weyl connections on the conformal structure defines a second order geometric structure (of first order) equivalent to the conformal structure.

The aim of this article is to introduce the almost conformally almost Fedosov structures as the second order geometric structures on manifolds with an ACS--structure that are analogies of the conformal Riemannian structures. However, it is more natural to define (see Definition \ref{def1.2}) these structures as the second order geometric structures on manifolds with an ACS--structure induced by a projective structure.

In this article, we use the Penrose's abstract index notation and Einstein summation conventions and we describe the ACS--structures using an analogy of the description of conformal Riemannian structures in \cite[Section 1.6]{parabook}. 

We start by fixing a symplectic form $J$ on $\mathbb{R}^{n}$ of the form $$((x_1,\dots,x_{n}),(y_1,\dots,y_{n}))\mapsto \sum_{i=1}^{\frac{n}{2}}x_iy_{\frac{n}{2}+i}-y_ix_{\frac{n}{2}+i}$$
and we denote by $CSp(n,\mathbb{R})$ the group of linear transformations preserving $J$ up to a multiple. The ACS--structure $(M,\ell)$ is equivalent to a $CSp(n,\mathbb{R})$--structure $\ba_0$, where the subbundle $\ba_0$ of the bundle of the frames of $TM$ consists of all frames in which the coordinates of sections of $\ell$ are of the form $fJ$ for some function $f$ on $M$ (depending on the frame). 

We view $J$ as a constant section $J_{ab}$ of a trivial line bundle $\ba_0\times \mathbb{R}$ and consider line bundles $\ba_0\times_{CSp(n,\mathbb{R})} \mathbb{R}[w]$, where $\mathbb{R}[w]$ is the $CSp(n,\mathbb{R})$--representation $A\mapsto (det(A))^{-\frac{w}{n}}$, as analogies of conformal densities. This 
identifies the sections of $\ell$ with the sections of  $\ba_0\times_{CSp(n,\mathbb{R})} \mathbb{R}[2]$ given by the above functions $f$.  Similarly, $J^{-1}$ defines a constant section $J^{bc}$ of $\ba_0\times \mathbb{R}$ and identifies sections of $\ell^*$ with sections of $\ba_0\times_{CSp(n,\mathbb{R})} \mathbb{R}[-2]$.

Let $U$ be a representation of $CSp(n,\mathbb{R})$. We will always represent the geometric objects on $M$ that are smooth sections of the bundles $\mathcal{U}:=\ba_0\times_{CSp(n,\mathbb{R})}U$ by the corresponding $CSp(n,\mathbb{R})$--equivariant functions from $\ba_0\to U$. For example, we represent the vector fields $\xi^a$ on $M$ by the $CSp(n,\mathbb{R})$--equivariant function $\ba\to \mathbb{R}^n$. Further, we view the linear connections $\nabla$ as operators mapping the sections of $\mathcal{U}$ onto the sections of $T^*M\otimes \mathcal{U}$ and for example, we write $\nabla_a\nu^d$ when $U=\mathbb{R}^n$. If we write $\mathcal{U}[w]$ for the associated bundle for the $CSp(n,\mathbb{R})$--representation $U\otimes \mathbb{R}[w]$, then for section $\xi^a$ of $TM$, $$\xi_b:=\xi^aJ_{ab}$$ is a section of $T^*M[-2]$ and, for section $\U_a$ of $T^*M$,  $$\U^d:=J^{db}\U_b$$ is a section of $TM[2]$. Analogously, we will use $J_{ab}$ and $J^{bc}$ to lower and rise the indices of general tensors. Let us emphasize that this always changes the conformal density and that we need to take in account the ordering of indices, because $$\delta^d{}_a=J^{db}\delta_b{}^cJ_{ca}=J_{ba}J^{db}=J_{ab}J^{bd}=-\delta_a{}^d,$$
where $\delta_i{}^j$ denotes the Kronecker delta.

Let us recall that \emph{a projective structure on $M$} is a class $[D]$ of torsion--free connections sharing the same geodesics (up to parametrization), cf. \cite[Section 4.1.5]{parabook}. For each $D,\bar D\in [D]$, there is a one--form $\U_a$ such that $$\bar D_a\nu^d=D_a\nu^d+(\U_a\delta_b{}^d+\U_b\delta_a{}^d)\nu^b.$$

We prove in Section \ref{sec3} that a connection $\nabla$ is an ACS--connection if and only if $\nabla_aJ_{bc}=0$. Clearly, the connections from the projective structure $[D]$ does not have to be ACS--connections. On the other hand, we can conclude from \cite[Theorem 1.1]{GRSFm} that for each torsion--free connection, there is an ACS--connection sharing the same geodesics (including the parametrization). Building on this result, we prove the following Theorem in Section \ref{sec3}.

\begin{thm}\label{thm1}
Two ACS--connections $\nabla$ and $\bar \nabla$ share the same geodesics (including the parametrization) if and only if there is a one--form $s_a$ such that 
$$\bar \nabla_a\nu^d:=\nabla_a\nu^d+(s_a \delta_{b}{}^d-s_b \delta_{a}{}^d-J_{ab}s^d)\nu^b.$$

Two ACS--connections $\nabla$ and $\bar \nabla$ share the same geodesics (up to parametrization) if and only if there are one--forms $s_a$ and $\beta_a$ such that 
$$\bar \nabla_a\nu^d:=\nabla_a\nu^d+((\beta_a+s_a) \delta_{b}{}^d+(\beta_b-s_b) \delta_{a}{}^d+J_{ab}(\beta^d-s^d))\nu^b.$$

Among the ACS--connections that share the same geodesics (up to parametrization) with a given projective class, there is a unique ACS--connection $\nabla^0$ with totally trace--free torsion.
\end{thm}

In other words, for each projective and ACS--structure on $M$, there is a unique ACS--connection $\nabla^0$ with totally trace--free torsion that shares the geodesics (up to parametrization) with the projective class. This connection describes the relation between the projective and ACS--structure. Therefore we get a geometric structure that generalizes \emph{the Fedosov structures}, which are a symplectic structures with a torsion--free symplectic connection, see \cite{GRSFm}.

\begin{def*}\label{def1.2}
We say that the triple $(M,\ell,\nabla^0)$ consisting of an ACS--structure $\ell$ and an ACS--connection $\nabla^0$ with totally trace--free torsion is \emph{an almost conformally almost Fedosov structure (shortly ACAF--structure).} A morphism of ACAF--structures $(M,\ell,\nabla^0)$ and $(M',\ell',(\nabla')^0)$ is a diffeomorphism $f: M\to M'$ such that $f^*\ell'=\ell$ and $f^*(\nabla')^0=\nabla^0$ hold.
\end{def*}

We prove in Section \ref{sec3} (see Theorem \ref{3.4}) that the ACAF--structure is equivalent in the categorical sense to a triple $(M,\ell,[D])$, where $[D]$ is the projective structure given by the geodesics of $\nabla^0$. Therefore we say that the triples $(M,\ell,[D])$ are an ACAF--structures, too. 

Up to this point, we did not assume any relations between $\ell$ and $[D]$. In particular,  the torsion of the connection $\nabla^0$ naturally decomposes into two components that measure this relations. This is the reason, why the word almost appears twice in the Definition \ref{def1.2}. The vanishing of the torsion in these components is related with the following subcategories of the category of ACAF--structures, see Lemma \ref{lem3.1}.

\begin{def*}
We call an ACAF--structure $(M,\ell,[D])$
\begin{enumerate}
\item \emph{a conformally almost Fedosov structure (shortly CAF--structure)} if $\ell$ admits local non--vanishing sections that are closed (i.e. symplectic forms),

\item \emph{an almost conformally Fedosov structure (shortly ACF--structure)} if there is an ACS--connection that shares the same geodesics (up to parametrization) with $[D]$ and with torsion given by the structure torsion of the ACS--structure,

\item \emph{a conformally Fedosov structure (shortly CF--structure)} if it is at the same time a CAF--structure and an ACF--structure.
\end{enumerate}
\end{def*}

Clearly, a CF--structure defines a Fedosov structure if $\ell$ admits a global closed non--vanishing section. In \cite{ES}, the authors used different definition of CF--structures, however it follows from Theorem \ref{3.4} that their definition is equivalent to our definition of CF--structures.

Now, let us observe that there is a class of ACS--connections from Theorem \ref{thm1} that satisfies an analogous transformation rule as the class of Weyl connections on a conformal Riemannian structure, cf. \cite[Section 1.6.4]{parabook}. This is a second order geometric structure that is closely related to the projective structure.

\begin{cor}
Let $(M,\ell,\nabla^0)$ be an ACAF--structure corresponding to a projective class $[D]$ on $M$. If $[\nabla^\beta]$ is the class of ACS--connections given by
$$\nabla^\beta_a\nu^d:=\nabla^0_a\nu^d+(\beta_a\delta_{b}{}^d+\beta_b\delta_{a}{}^d+J_{ab}\beta^d)\nu^b$$
for all one--forms $\beta_a$, then there is a bijection between the class $[D]$ and $[\nabla^\beta]$ that assigns to connection in $[D]$ a connection in $[\nabla^\beta]$ that shares the same geodesics (including the parametrization).
\end{cor}

We say that the class $[\nabla^\beta]$ from the above corollary is \emph{the class of Weyl connections on the ACAF--structure}. 

The class of the Weyl connections on a conformal Riemannian structure provides a Cartan (parabolic) geometry  that solves the equivalence problem for the conformal Riemannian structures and has many other applications, cf. \cite{parabook}. The projective structure $[D]$ itself provides a Cartan (parabolic) geometry, too. In the case of ACAF--structures, the distinguished connection $\nabla^0$ and the class $[\nabla^\beta]$ of Weyl connections provide different Cartan geometries with different applications.

Firstly, there is a Cartan connection $\om^0$ of type $(\mathbb{R}^n\rtimes CSp(n,\mathbb{R}),CSp(n,\mathbb{R}))$ with totally trace--free torsion on the bundle $\ba_0$, where we assume that $\mathbb{R}^n$ is the standard representation of $CSp(n,\mathbb{R})$. This is the Cartan geometry given by the ACS--connection $\nabla^0$, see \cite[Section 1.6.1]{parabook}. The Cartan geometry $(\ba_0\to M,\om^0)$ of type $(\mathbb{R}^n\rtimes CSp(n,\mathbb{R}),CSp(n,\mathbb{R}))$ can be used to solve the equivalence problem for ACAF--structures, to compute all (infinitesimal) automorphism and invariants of the ACAF--structures. Moreover, it makes possible to construct all invariant differential operators between natural vector bundles associated with ACAF--structures.

Secondly, there is a Cartan geometry $(\ba\to M,\om)$ of type $(\fl,P,\Ad)$, where $P$ is a (parabolic) subgroup of $Sp(n+2,\mathbb{R})$ stabilizing an isotropic line for standard action on $\mathbb{R}^{n+2}$, $\fl$ is a condimension one $P$--invariant subspace of $\frak{sp}(n+2,\mathbb{R})$ and $\Ad$ is the restriction of the adjoint representation of $Sp(n+2,\mathbb{R})$ to $P$ and $\fl$. This Cartan geometry is not modeled on a Klein geometry, but modeled on a skeleton $(\fl,P,\Ad)$, see Appendix \ref{ApS}.  We fix a symplectic form on $\mathbb{R}^{n+2}$ of the form 
$$((x_0,\dots,x_{n+1}),(y_0,\dots,y_{n+1}))\mapsto x_0y_{n+1}-y_0x_{n+1}+\sum_{i=1}^{\frac{n}{2}}x_iy_{\frac{n}{2}+i}-y_ix_{\frac{n}{2}+i}$$
and denote by $Sp(n+2,\mathbb{R})$ the group of linear transformations preserving this symplectic form. Note that this symplectic form restricts to $J$ on the subspace $(0,x_1,\dots, x_n,0)$. We fix $P$ in $Sp(n+2,\mathbb{R})$ as the stabilizer of the line given by $x_0$. Let us recall that $P$ is parabolic subgroup of $Sp(n+2,\mathbb{R})$ associated to the following contact grading of 
\begin{align*}
\frak{sp}(n+2,\mathbb{R})=\begin{pmatrix}
\fp_{0}& \fp_{1} & \fp_{2} \cr
\fl_{-1}& \fp_{0} & \fp_{1}\cr
\fg_{-2}& \fl_{-1} & \fp_{0} \cr
\end{pmatrix}:=
\begin{Bmatrix}
\begin{pmatrix}
a& Y_c &  z \cr
X^d& A_c{}^d & Y^d\cr
x& -X_c & -a \cr
\end{pmatrix},
A_c{}^d =A_d{}^c,\ x,a,z\in \mathbb{R}
\end{Bmatrix}.
\end{align*}
Let us point out that we view $CSp(n,\mathbb{R})$ as the subgroup of $P$ preserving this grading and that the grading has form $\frak{sp}(n+2,\mathbb{R})=\mathbb{R}[-2]\oplus \mathbb{R}^n\oplus \frak{csp}(n,\mathbb{R})\oplus (\mathbb{R}^n)^*\oplus \mathbb{R}[2]$ as a $CSp(n,\mathbb{R})$--module. Further, we consider the $CSp(n,\mathbb{R})$--invariant decompositions $\fl=\fl_{-1}\oplus \fp_0\oplus \fp_1\oplus \fp_2$ and $P= CSp(n,\mathbb{R})\exp(\fp_1)\exp(\fp_2)$. In Section \ref{sec4}, we describe the relations between the Cartan geometries of type $(\fl,P,\Ad)$, the Cartan geometries of type $(\mathbb{R}^n\rtimes CSp(n,\mathbb{R}),CSp(n,\mathbb{R}))$ and the ACAF--structures.

Since $P$ is a parabolic subgroup of $Sp(n+2,\mathbb{R})$, many of the properties and constructions from the theory of parabolic geometries described in \cite{parabook} carry over to the case of Cartan geometries of type $(\fl,P,\Ad)$. In particular, we can carry over the theory of Weyl structures from \cite[Section 5.1]{parabook}, because for a Cartan geometry $(\ba\to M,\om)$ of type $(\fl,P,\Ad)$, we identify in the Section \ref{sec4} the quotient $\ba/\exp(\fp_1)\exp(\fp_2)$ with the underlying $CSp(n,\mathbb{R})$--bundle $\ba_0$. Then a $CSp(n,\mathbb{R})$--equivariant section $\si: \ba_0 \to \ba$ is called \emph{a Weyl structure} of the Cartan geometry $(\ba\to M,\om)$ of type $(\fl,P,\Ad)$ and we say that
\begin{enumerate}
\item the component $\si^*\om_-$ of $\si^*\om$ with values in $\fl_{-1}$ is \emph{a soldering form},
\item the component $\si^*\om_0$ of $\si^*\om$ with values in $\fp_0$ is \emph{a Weyl connection}, and
\item the components $\si^*\om_1\oplus \si^*\om_2$ of $\si^*\om$ with values in $\fp_1\oplus \fp_2$ are \emph{Rho tensors}.
\end{enumerate}

The Weyl structures allow us to characterize the subcategory of the category of Cartan geometries of type $(\fl,P,\Ad)$ that is equivalent to the category of ACAF--structures. Let us emphasize that this subcategory is not unique and depends on a so--called normalization conditions.

\begin{thm}\label{1.6}
There is an equivalence of categories between the category of ACAF--structures on $M$ and the category of Cartan geometries $(\ba\to M,\om)$ of type $(\fl,P,\Ad)$ that admit a Weyl structure $\si^0: \ba_0\to \ba$ such that
\begin{enumerate}
\item the Weyl connection of $\si^0$ has totally trace--free torsion, and
\item the Rho--tensors of $\si^0$ vanish identically.
\end{enumerate}

In particular, for such a Cartan geometry, the class of Weyl connections of all Weyl structures is precisely the class of Weyl connections on the ACAF--structure.
\end{thm}

We will see that we can prescribe different values to the Rho--tensors of the distinguished Weyl structure $\si^0$ and obtain different normalization conditions.

As in \cite[Section 5.1]{parabook}, any Weyl structure $\si$ provides a splitting, that is, an isomorphism $\ba\times_P U\cong \ba_0\times_{CSp(n,\mathbb{R})}U$, of each associated bundle for a representation $U$ of $P$ representations to (sum of) associated bundles for the induced $CSp(n,\mathbb{R})$--representations. In this article, we use the distinguished nature of the Weyl structure $\si^0$ and identify $\ba\times_P U= \ba_0\times_{CSp(n,\mathbb{R})}U$ using the splitting provided by the Weyl structure $\si^0$. A consequence of the Theorem \ref{1.6} is that this identification is natural in the way that depends only on the ACAF--structure and is preserved by all automorphisms of the ACAF--structure.

We show in Section \ref{sec5} that for Cartan geometries $(\ba\to M,\om)$ of type $(\fl,P,\Ad)$, we can construct analogies of the Bernstein-Gelfand-Gelfand (shortly BGG)--sequences and BGG--complexes on the projective and conformal structures, cf. \cite{BGG}. In fact, we adapt the construction of BGG--sequences and BGG--complexes from \cite{CSou} for the Cartan geometries of type $(\fl,P,\Ad)$. Let us emphasize that since it enough to work in the splitting provided by the distinguished Weyl structure $\si^0$, we can construct many more sequences (and sometimes complexes) of invariant differential operators than in \cite{CSou}. Therefore we call them BGG--like sequences and BGG--like complexes.

In Section \ref{sec6}, we provide examples of BGG--like sequences that exist on all ACAF--structures and that become BGG--like complexes on flat ACAF--structures. In particular, on CF--structures, we construct some of these examples between the same bundles and with the same symbol as the descended BGG--sequences in \cite{CSaII}, however in general, these can differ from our examples by an invariant operator of lower order. Further, we show how we can construct the BGG complexes from \cite{ES} and the descended BGG--complexes from \cite{CSaII} that exist on CF--structures with particular curvature (curvature of Ricci type).

There are further possible applications of the Cartan geometry of type $(\fl,P,\Ad)$. As in the case of projective and conformal Riemannian structures, the first differential operator in the BGG--like sequence is an overdetermined differential operator and the BGG--machinery for the Cartan geometries of type $(\fl,P,\Ad)$ prolongs this operator to a closed form (a system of first order ODEs). For example, we can conjecture that one of these operators should give a solution to the question, whether there is a Weyl connection on ACAF--structure preserving a complex structure. This means that the theory of the holonomy reductions of the Cartan geometries from \cite{CG} should have a reasonable adaptation for the Cartan geometries of type $(\fl,P,\Ad)$. We conjecture that such a holonomy reduction for a subgroup $K\subset Sp(n+2,\mathbb{R})$ should decompose $M$ into the so--called curved orbits depending on certain $K$--orbits in $Sp(n+2,\mathbb{R})/P$. Moreover, an orbit corresponding to $g\in Sp(n+2,\mathbb{R})$ should carry a Cartan geometry of type $(\fl\cap \Ad(g)\frak{k},P\cap \Ad(g)K,\Ad)$.

\section{Connections on ACAF--structures}\label{sec3}

Let us start with an ACAF--structure $(M,\ell,[D])$ and view $J_{bc}$ as a section of $\wedge^2T^*M[-2]$. We can use the representation theory of the group $CSp(n,\mathbb{R})$ and decompose $D_aJ_{bc}$ to irreducible components of $T^*M\otimes\wedge^2T^*M[-2]$, see \cite{AP} for details. In general, we get the following decomposition
$$D_a J_{bc}=2\alpha_a J_{bc}+J_{ab}\beta_c-J_{ac}\beta_b-H_{bca}+2S_{bca},$$
where $S_{b(ca)}=0,S_{bc}{}^b=0,H_{[bca]}=0,H_{bc}{}^b=0.$ Since $D_a(zJ)_{bc}=D_azJ_{bc}+zD_aJ_{bc}$ holds for section $z$ of $\ba_0\times_{CSp(n,\mathbb{R})}\mathbb{R}[2]$, we see that $D$ in the projective class is an ACS--connection if and only if
$$D_aJ_{bc}=2\alpha_aJ_{bc}$$
holds for the induced linear connection on $\wedge^2T^*M[-2]$ and some one--form $\alpha_a$.

An element $s_{ab}{}^d$ acts (as a one--form valued in the endomorphisms of $TM$) on $J_{bc}$ as $-J_{dc}s_{ab}{}^d-J_{bd}s_{ac}{}^d+\frac{2}{n}s_{ad}{}^dJ_{bc}$. Therefore if $\bar D_a\nu^d=D_a\nu^d+(\U_a\delta_b{}^d+\U_b\delta_a{}^d)\nu^b$ is substituted in the formula for $D_a J_{bc}$, then $H_{bca},\ S_{bca}$ remain the same and the one--forms $\alpha, \beta$ are subject to the change
$$\bar \alpha_a=\alpha_a+\frac1{n}\U_a, \bar \beta_a=\beta_a+\U_a.$$ In particular, the one--form $\beta$ allows us to distinguish between the linear connections in the class $[D]$.

\begin{def*}
Let $(M,\ell,[D])$ be an ACAF--structure.  For each one--form $\beta_a$, we denote by $D^\beta$ the unique linear connection in $[D]$ satisfying
$$D^\beta_a J_{bc}=2\alpha_a J_{bc}+J_{ab}\beta_c-J_{ac}\beta_b-H_{bca}+2S_{bca}.$$
\end{def*}

Now we can  use the construction from \cite[Theorem 1.1]{GRSFm} to obtain all ACS--connections that share the geodesics (up to parametrization) with the projective class $[D^\beta]$.

\begin{prop}\label{propconconst}
Let $(M,\ell,[D^\beta])$ be an ACAF--structure. The linear connections
$$\nabla_a^{\beta,s}\nu^d:=D^\beta_a\nu^d+(s_a \delta_{b}{}^d-s_b \delta_{a}{}^d-s^d J_{ab}+H_{ab}{}^d+S_{ab}{}^d+J_{ab}\beta^d)\nu^b$$
with the torsion
$$(T^{\beta,s})_{ab}{}^d=2(s_a \delta_{b}{}^d-s_b \delta_{a}{}^d-s^dJ_{ab}+H_{ab}{}^d+S_{ab}{}^d+J_{ab}\beta^d)
$$
for all one--forms $s_a$ exhaust all ACS--connections on $M$ share the same geodesics (including the parametrization) with $D^\beta$. Moreover, $\nabla_a^{\beta,s}J_{bc}=0$ and $\alpha_a=\frac1n\beta_a$.

Therefore all ACS--connections that share the same geodesics (up to parametrization) with $[D^\beta]$ are of the form
$\nabla^{\beta,s}$ for all one--forms $\beta,s$. In particular, $$\nabla^{0}:=\nabla^{0,0}$$ is the unique ACS--connection with a totally trace--free torsion and that shares the same geodesics (up to parametrization) with $[D^\beta]$.
\end{prop}
\begin{proof}
If we define the following tensor field
\begin{align*}
(F^{\beta,s})_{ab}{}^d:&=\frac12(2s_a \delta_{b}{}^d+J_{ab}\beta^d-\delta_a{}^d\beta_b-H_b{}^d{}_a+2S_b{}^d{}_a),
\end{align*}
then the linear connection $\nabla^{\beta,s}$ can be equivalently written by the formula
$$\nabla_a^{\beta,s}\nu^d=D^\beta_a\nu^d+((F^{\beta,s})_{ab}{}^d-(F^{\beta,s})^{d}{}_{ab}-(F^{\beta,s})_{ba}{}^d)\nu^b.$$

The element $(F^{\beta,s})_{ab}{}^d$ maps $J_{bc}$ onto $$J_{cd}(F^{\beta,s})_{ab}{}^d-J_{bd}(F^{\beta,s})_{ac}{}^d+\frac{2}{n}(F^{\beta,s})_{ad}{}^dJ_{bc}=-(F^{\beta,s})_{abc}+(F^{\beta,s})_{acb}+\frac{2}{n}(F^{\beta,s})_{ad}{}^dJ_{bc}$$ and thus
$(F^{\beta,s})^{d}{}_{ab}$ maps $J_{bc}$ onto $-(F^{\beta,s})_{cab}+(F^{\beta,s})_{bac}-\frac{2}{n}(F^{\beta,s})_{da}{}^dJ_{bc}$. Therefore
$$\nabla^{\beta,s}_a J_{bc}
=D^\beta_a J_{bc}-2(F^{\beta,s})_{abc}+(2s_a-\frac2n\beta_a)J_{bc}=(2\alpha_a-\frac2n\beta_a) J_{bc}$$
holds. Therefore $\nabla^{\beta,s}_a$ is an ACS-connection and $\nabla^{\beta,s}_a J_{bc}=0$, because $J_{bc}$ is a constant section of trivial bundle. Since the connections $\nabla^{\beta,s}$ are ACS--connections that differ from $D^\beta$ by sections of $\wedge^2T^*M\otimes TM$, the connections $\nabla^{\beta,s}$ and $D^\beta$ share the same geodesics (including the parametrization).

The symmetric part of the difference $$\nabla_a^{\beta+\U,s} \nu^d-\nabla_a^{\beta,s} \nu^d
 =(\U_a\delta_b{}^d+\U_b\delta_a{}^d+J_{ab}\U^d)\nu^b$$ is exactly the projective change $D^{\beta+\U}_a\nu^d-D^\beta_a\nu^d=(\delta_a{}^{d}\U_b+\U_a\delta_b{}^{d})\nu^b$ and thus it remains to prove that we have already found all ACS--connections that share the same geodesics  (up to parametrization) with $[D^\beta]$.

The change between two ACS--connections is a section of $\ba\times_{CSp(n,\mathbb{R})}(\mathbb{R}^n)^*\otimes \frak{csp}(n,\mathbb{R})$. Since $(\mathbb{R}^n)^*\otimes \frak{sp}(n,\mathbb{R})\cong S^2(\mathbb{R}^n)^*\otimes \mathbb{R}^n$, two ACS--connections that share the same geodesics (including the parametrization) differ by a one--form and the claim follows, because we can identify this one--form with the one--form $s_a$.
\end{proof}

A simple computation shows that the Theorem \ref{thm1} is a direct consequence of the above proposition. Let us discus the role of the components $S_{bc}{}^d,\ H_{bc}{}^d$ of the torsion of $\nabla^0$.

\begin{lem}\label{lem3.1}
\begin{enumerate}
\item The section $S_{bc}{}^d$ of $\wedge^2T^*M\otimes TM$ is the structural torsion of the ACS--structure $\ell$, i.e., it is an obstruction for $\ell$ to be a conformally symplectic structure (and for the ACAF--structure to be a CAF--structure).
\item The section $H_{bc}{}^d$ of $\wedge^2T^*M\otimes TM$ provides an obstruction for the ACAF--structure to be an ACF--structure.
\end{enumerate}
\end{lem}
\begin{proof}
Since the trace--free part of $(dJ)_{abc}=D_a J_{bc}+D_b J_{ca}+D_c J_{ab}$ is equal to $6S_{abc}$, the torsion $S$ is the structural torsion of the ACS--structure $\ell$. The claim relating vanishing of $H$ with ACF--structures is a consequence of the Proposition \ref{propconconst}.
\end{proof}

Now we can prove the equivalence of categories between the category of ACAF--structures and the category of triples $(M,\ell,[D])$ given by a ACS--structure and a projective structure.

\begin{thm}\label{3.4}
Suppose $(M,\ell,[D])$ and $(M',\ell',[D'])$ correspond to ACAF--structures $(M,\ell,\nabla^0)$ and $(M',\ell',(\nabla')^0)$. Then a (local) diffeomorphism $f:M\to M'$ is a projective and ACS--morphism between $(M,\ell,[D])$ and $(M',\ell',[D'])$ if and only if it is ACAF--morphism between $(M,\ell,\nabla^0)$ and $(M',\ell',(\nabla')^0)$.
\end{thm}
\begin{proof}
It holds $f^*\ell'=\ell$ if and only if $f$ is an ACS--morphism. If $f^*(\nabla')^0=\nabla^{0}$ holds, then $f$ maps the geodesics of $\nabla^{0}$ onto the geodesics of $(\nabla')^0$ and thus it is a projective morphism by construction of $\nabla^{0},(\nabla')^0$.

If $f$ is an ACAF--morphism, then we know that $f^*(\nabla')^0$ is an ACS--connection that has totally trace--free torsion and that shares the same geodesics (up to parametrization) with $[D]$. Thus the equality $f^*(\nabla')^0=\nabla^{0}$ follows from the uniqueness claim of the Proposition \ref{propconconst}.
\end{proof}

\section{Weyl structures on Cartan geometries of type $(\fl,P,\Ad)$ and proof of Theorem \ref{1.6}}\label{sec4}

Let us provide more results about the Weyl structures on the Cartan geometries of type $(\fl,P,\Ad)$. In particular, let us provide all the formulas for the change of the splittings for the representation $Ad$ of $P$ on $\fl$ for general Weyl structures. The formulas for the changes of splittings for the general representation can be found in \cite[Section 5.1]{parabook} and we will not need them explicitly due to the existence of the distinguished Weyl structure $\si^0$.

\begin{prop}
The following statements hold for a Cartan geometry $(\ba\to M,\om)$ of type $(\fl,P,\Ad)$:
\begin{enumerate}
\item There exists a global Weyl structure $\si: \ba/\exp(\fp_1)\exp(\fp_2)\to \ba.$ 
\item There is a unique $CSp(n,\mathbb{R})$--equivariant inclusion $\iota$ of $\ba/\exp(\fp_1)\exp(\fp_2)$ into the first order frame bundle on $M$ such that $\si^*\om_-=\iota^*\theta$ holds for the natural soldering form $\theta$ on the frame bundle and all Weyl structures $\si$. In particular, the image $\ba_0=\iota(\ba/\exp(\fp_1)\exp(\fp_2))$ does not depend on $\si$, defines an underlying ACS--structure $(M,\ell)$ and identifies $\ba_0=\ba/\exp(\fp_1)\exp(\fp_2)$.
\item Fixing one Weyl structure $\si$, there is bijective correspondence between the set of all Weyl structures and the space of smooth sections of $T^*M\oplus \ell$. Explicitly, this correspondence is given by mapping smooth section $\U_a+ yJ_{ab}$ of $T^*M\oplus \ell$ to the Weyl structure $$\hat \si=\si\exp(\U_a)\exp(y).$$
\item The Weyl connection $\si^*\om_0$ is a principal connection on the principal bundle $\ba_0\to M$ and if we denote by $\hat \nabla_a\nu^b,\nabla_a\nu^b$ the induced Weyl connections on $TM$ corresponding to the forms  $\hat \si^*\om_0,\si^*\om_0$, then 
$$\hat \nabla_a\nu^d=\nabla_a\nu^d+(\delta_a{}^{d}\U_b+\U_a\delta_b{}^{d}+\U^dJ_{ab})\nu^b.$$
\item The Rho tensor $\si^*\om_1$ is a tensor $\Rho_{ab}$, the Rho tensor $\si^*\om_2$ is an $\ell$--valued one--form $\Rho_{a}J_{bc}$ and
$$\hat \Rho_{ab}=\Rho_{ab}-\U_a\U_b+\nabla_a\U_b+y J_{ab},$$
$$\hat \Rho_{a}=\Rho_{a}+\nabla_ay+2\Rho_{ab}\U^b+\nabla_a \U_b\U^b-2\U_ay$$
holds for the Rho tensors of the Weyl structures $\hat \si^*,\si^*$
\end{enumerate}
\end{prop}
\begin{proof}
The proof of the Claim (1) \cite[Proposition 5.1.1]{parabook} in the case of parabolic geometries can be directly used to prove our Claim (1). For the proof of the remaining claims, we can use the relevant proofs of Propositions in \cite[Section 5.1]{parabook}, because these require only the information provided by the skeleton. Therefore $\si^*\om_-, \nabla_a\nu^b, \Rho_1, \Rho_2$ transform according to the same transformation formulas as in the case of parabolic geometries and we just rewrite them using our notations and conventions. In particular, $\si^*\om_-$ does not depend on the choice of the Weyl structure and thus $\iota$ is clearly induced by the identification of $\fl_{-1}$ with $\mathbb{R}^n$. Then it is obvious that $T^*M\cong \ba_0\times_{CSp(n,\mathbb{R})} \fp_1$ and the identification $\ell\cong \ba_0\times_{CSp(n,\mathbb{R})} \fp_2$ follows from the Claim (2).
\end{proof}

We see from the formula for the change of the Weyl connections that they all share the same geodesics (up to parametrization) and thus they define an underlying ACAF--structure.

\begin{prop}
For each Cartan geometry $(\ba\to M,\om)$ of type $(\fl,P,\Ad)$, there is a unique ACAF--structure $(M,\ell,[D])$ such that $\ell$ is the $CSp(n,\mathbb{R})$--structure $\ba_0$ and $[D]$ is the projective structure that shares the geodesics (up to parametrization) with some (and thus each) of the Weyl connections.

This assignment is a faithful functor from the category of Cartan geometries of type $(\fl,P,\Ad)$ to the category of ACAF--structures, i.e., if $\phi_1,\phi_2$ are two morphism between the Cartan geometries of type $(\fl,P,\Ad)$, which coincide as ACAF--morphisms, then $\phi_1=\phi_2$.
\end{prop}
\begin{proof}
The above results imply that it remains to deal with the morphisms. Clearly each morphism of the Cartan geometries of type $(\fl,P,\Ad)$ preserves $\si^*\om_-$ and thus it is an ACS--morphism. Secondly, a morphism of Cartan geometries of type $(\fl,P,\Ad)$ maps Weyl connections onto Weyl connections and thus it preserves their geodesics (up to parametrization), i.e., it is a projective morphism. To prove the faithfulness, we consider the composition $\phi=\phi_1^{-1}\phi_2$, which is a locally defined automorphism of a Cartan geometry of type $(\fl,P,\Ad)$. Since $\phi$ acts as $\id$ on the projective structure by assumption, the underlying diffeomorphism on $M$ is $\id$ and $\phi=\id$ follows from the Proposition \ref{rigid}.
\end{proof}

Let us recall the well--known relation (cf. \cite[Section 1.6.1]{parabook}) between the linear connections on $G$--structures and the Cartan geometries of type $(\mathbb{R}^n\rtimes G,G)$. This means that there is an equivalence of categories between the category of  ACAF--structures and the category of Cartan geometries $\om^0$ on $\ba_0$ of type $(\mathbb{R}^n\rtimes CSp(n,\mathbb{R}),CSp(n,\mathbb{R}))$ with totally trace--free torsion. 

Let us start proving the Theorem \ref{1.6}. We consider the extension functors from \cite[Theorem 1.5.15]{parabook} and \cite[Theorem 1.3]{ja-arx}. It is a simple observation that the inclusion $CSp(n,\mathbb{R})\subset P$ and $\mathbb{R}^n\oplus \frak{csp}(n,\mathbb{R})\subset \fl$ given by identification $\mathbb{R}^n=\fl_{-1}$ satisfy all the conditions of \cite[Theorem 1.5.15]{parabook} and \cite[Theorem 1.3]{ja-arx} and thus there is an extension functor $\mathcal{F}$ from the category of Cartan geometries of type $(\mathbb{R}^n\rtimes CSp(n,\mathbb{R}),CSp(n,\mathbb{R}))$ to the category of Cartan geometries of type $(\fl,P,\Ad)$. Let us recall that for a Cartan geometry $(\ba_0\to M,\om^0)$ of type $(\mathbb{R}^n\rtimes CSp(n,\mathbb{R}),CSp(n,\mathbb{R}))$ and the morphism $\phi_0$ from the Cartan geometry $(\ba_0\to M,\om^0)$, the following holds:
$$\mathcal{F}(\ba_0\to M)=\ba_0\times_{CSp(n,\mathbb{R})}P\to M$$
$$\mathcal{F}(\om^0)|_{T_{(u_0,e)}\mathcal{F}(\ba_0)}=\om^0(u_0)+\om_P(e)$$
$$\mathcal{F}(\phi_0)(u_0,e)=(\phi_0(u_0),e),$$
where $\om_P$ is the Maurer--Cartan form on $P$ and $(u_0,e)$ is the class in $\ba_0\times_{CSp(n,\mathbb{R})}P$ of the point $u_0\in \ba_0$ and the identity element $e\in P$.

To finish the proof of the Theorem \ref{1.6}, it remains to check that the image of $\mathcal{F}$ is characterized by the claimed normalization condition.

\begin{lem}
A Cartan geometry $(\ba\to M,\om)$ of type $(\fl,P,\Ad)$ with the underlying Cartan geometry $(\ba_0\to M,\om^0)$ of type $(\mathbb{R}^n\rtimes CSp(n,\mathbb{R}),CSp(n,\mathbb{R}))$ with totally trace--free torsion is isomorphic to $\mathcal{F}(\ba_0\to M,\om^0)$ if and only if there is a (unique) Weyl structure $\si^0: \ba_0\to \ba$ such that
\begin{enumerate}
\item the Weyl connection of $\si^0$ has totally trace--free torsion, and
\item the Rho--tensors of $\si^0$ vanish identically.
\end{enumerate}
\end{lem}
\begin{proof}
It is clear that the map $\ba_0\to \mathcal{F}(\ba_0)$ defined as $u_0\mapsto (u_0,e)$ is a Weyl structure with the same properties as $\si^0$. Thus if $\si^0$ has the claimed properties, then the map $\si^0(u_0)p\mapsto (u_0,p)$ for all $u_0\in \ba_0$ and $p\in P$ is clearly an isomorphism of Cartan geometries $(\ba\to M,\om)$ and $\mathcal{F}(\ba_0\to M,\om^0)$.

Conversely, suppose $\phi: \mathcal{F}(\ba_0)\to \ba$ is an isomorphism of Cartan geometries. Then $\si^0(u_0)=\phi(u_0,e)$ is a Weyl structure on $(\ba\to M,\om)$. We can assume that $\phi$ covers identity on $M$. Thus $(\phi^{-1})^*\si^0$ is the Weyl structure $u_0\mapsto (u_0,e)$ and the Weyl connections and the Rho--tensors of $\si^0$ are pull--backs by the identity map, i.e., they have the claimed properties.
\end{proof}

\section{BGG--machinery for Cartan geometries of type $(\fl,P,\Ad)$}\label{sec5}

In this section, we adapt the construction of the BGG--sequences from \cite{CSou} for the Cartan geometries of type $(\fl,P,\Ad)$. This will be done in two steps. 

In the first step, we construct objects with the properties that are necessary for the construction from \cite{CSou}. In particular, we need to find analogies of tractor bundles, tractor bundle valued forms on which we can find algebraic differentials and codifferentials. The reason is that the construction from \cite{CSou} can not be applied to the direct generalization of the setting in the case of parabolic geometries to our situation. On the other hand, we will find that in fact, there are more possible settings in which the construction from \cite{CSou} can be applied.

In the second step, we show that in the setting we find in the first step, we can directly construct the BGG--like sequences by applying the construction of the BGG--sequences from \cite{CSou}.

We illustrate, how these steps look like in a particular setting on an example in Appendix \ref{ApC}. Let us note that many of the results in the following sections can be directly obtained from this example using the representation theory.

We consider a Cartan geometry $(\ba\to M,\om)$ of type $(\fl,P,\Ad)$ with a distinguished Weyl structure $\si^0$, without assuming any relation of $\si^0$ with the underlying ACAF--structures or imposing any normalization conditions. Nevertheless, the splittings provided by Weyl structure $\si^0$ remain natural and invariant w.r.t. all automorphisms, which allows us to extend each $CSp(n,\mathbb{R})$--invariant object to a $P$--invariant object. This is the main reason we get more possible settings for the BGG--machinery.

\subsection{Setting the BGG--machinery}

We consider the following generalizations of tractor bundles and tractor valued forms.

\begin{def*}
Let $\tau: Sp(n+2,\mathbb{R})\to GL(W)$ be a representation of $Sp(n+2,\mathbb{R})$ on a vector space $W$ and let $V\subset W$ be a $\tau(P)$--invariant subspace. We say that the associated bundle  $\ba\times_P W$ is \emph{a tractor bundle} on the ACAF--structure and we say that the associated bundle $\ba\times_P V$ is \emph{a tractor--like bundle} on the ACAF--structure. 
\begin{enumerate}
\item In the space of $V$--valued, we fix a subspace
$$L^i(V):=\wedge^{i}_0(\fl/\fp)^*\otimes V+\wedge^{i}(\fl/\fp)^*\otimes Ker(d\tau(\fp_2)),$$
where $\wedge^{i}_0(\fl/\fp)^*$ is the space the trace--free forms and $Ker(d\tau(\fp_2))$ is the subspace of $V$ annihilated by action $d\tau(X)$ for all $X\in \fp_2$.
\item In the space of $V[2]$--valued, we fix a subspace
$$L^{i}(V[2]):=\wedge^{i}_0(\fl/\fp)^*\otimes V[2]+\wedge^{i}(\fl/\fp)^*\otimes Ker(d\tau(\fp_2))[2].$$
\item We combine the spaces of all $W$--valued and $W[2]$--valued forms into the space
$$\wedge^i \fp_+\otimes W$$
using the $CSp(n,\mathbb{R})$--module isomorphism
$$\fp_1\oplus \fp_2\cong \fp_+:= (\fl/\fp)^*\oplus \mathbb{R}[2].$$
\end{enumerate}
We use the notation $\Omega^i(L^i(V)),\Omega^i(L^i(V[2])),\Omega^i(W),$ for the spaces of sections of $\ba\times_P L^i(V)$, $\ba\times_P L^i(V[2])$ or $\ba\times_P \wedge^i(\fp_+\otimes W)$. We write just $\Omega^i$ for these spaces in the case it is not necessary to distinguish between them.
\end{def*}

Note that as the $P$--modules $L^i(V), L^i(V[2]),\wedge^i(\fp_+\otimes W)$ are filtered (via action of $d\tau(\fp_+)$) and the Weyl structure $\si^0$ provides splitting of each of the filtrations, that is, an isomorphism with the associated grading. In particular, each bundle $\Omega^i$ is naturally filtered.

The next ingredients for the construction of the BGG--sequences in \cite{CSou} are $CSp(n,\mathbb{R})$--equivariant Lie algebra differentials $\partial: \wedge^i \fp_+\otimes W\to \wedge^{i+1} \fp_+\otimes W$ and $P$--equivariant Konstant's codifferentials $\partial^*: \wedge^i \fp_+\otimes W\to \wedge^{i-1} \fp_+\otimes W$. Let us recall from \cite[Section 3.1.10]{parabook} that $\partial$ and $\partial^*$ are defined by the formulas
\begin{align*}
(\partial\phi)&((X_1+x_1),\dots, (X_{i+1}+x_{i+1})):=\sum_k (-1)^k d\tau(X_k+x_k)\phi((X_1+x_1),\hat\dots, \\
&(X_{i+1}+x_{i+1}))+\sum_{k<j}(-1)^{k+j}\phi(\{X_k,X_j\},(X_1+x_1),\hat\dots, (X_{i+1}+x_{i+1})),
\end{align*}
\begin{align*}
\partial^*&((Z_1+z_1)\wedge\dots\wedge (Z_i+z_i)\otimes v):=\sum_k (-1)^k (Z_1+z_1)\wedge\hat\dots\wedge (Z_i+z_i)\otimes\\
&d\tau(Z_k+z_k)v+\sum_{k<j}(-1)^{k+j}[Z_k,Z_j]\wedge (Z_1+z_1)\wedge\hat \dots\wedge (Z_i+z_i) \otimes v 
\end{align*}
for $\phi\in \wedge^i \fl_{-1}^*\otimes W\oplus \wedge^i \fl_{-1}^*\otimes \mathbb{R}[-2]^*\otimes W, X_k\in \fl_{-1}\cong \fp_+^*, x_k\in \mathbb{R}[-2]\cong \fp_2^*,Z_k\in \fp_1,z_k\in \fp_2, v\in W$, where $\{.,.\}: \wedge^2\fl_{-1}^*\otimes \mathbb{R}[-2]$ is dual to bracket $[.,.]:\wedge^2 \fp_1^*\otimes \fp_2$ and $\hat \dots$ always means that we omit  in the sequence the elements $Z_l+z_l$ for all summation indices $l$.

We consider $V$--valued forms and denote by $\iota$ and $r$ the inclusions and projections of $V$--valued forms into $\wedge^i(\fp_+\otimes W)$, which exist, because $\fp_1\cong (\fl_{-1})^*\cong (\fl/\fp)^*$ as $CSp(n,\mathbb{R})$--modules and the representations of $CSp(n,\mathbb{R})$ on $W$ induced by $\tau$ are always completely reducible, i.e., there is $CSp(n,\mathbb{R})$--invariant complement of $V$ in $W$. Then we get the following diagram of $CSp(n,\mathbb{R})$--equivariant maps
\begin{displaymath}
\xymatrix{
0& W \ar[l]_{\partial^*} \ar[d]_{r}& \fp_+\otimes W \ar[l]_{\partial^*}\ar[d]_{r}& \dots\ar[l]_{\partial^*}\ar[d]_{r} & \wedge^n\fp_+\otimes W \ar[l]_{\partial^*}\ar[d]_{r}& \wedge^{n+1}\fp_+\otimes W \ar[l]_{\partial^*} \\
0& V \ar[l]_{\partial^*_0} \ar@/_/[u]_{\iota}& (\fl/\fp)^*\otimes V \ar[l]_{\partial^*_0}\ar@/_/[u]_{\iota}& \dots\ar[l]_{\partial^*_0}\ar@/_/[u]_{\iota} & \wedge^n(\fl/\fp)^*\otimes V \ar[l]_{\partial^*_0}\ar@/_/[u]_{\iota}& 0 \ar[l]\\
}
\end{displaymath}
where $$\partial^*_0:=r\circ \partial^*\circ \iota.$$ This means that
\begin{align*}
\iota\circ \partial^*_0\circ r((Z_1+z_1)\wedge\dots\wedge (Z_i+z_i)\otimes v)&=\sum_k (-1)^k Z_1\wedge\hat \dots\wedge Z_i\otimes d\tau(Z_k)v,\\
\end{align*}
\begin{align*}
(\partial^*_0)^2(Z_1\wedge\dots\wedge Z_i\otimes v)&=r\circ (\iota\circ \partial^*_0\circ r-\partial^*)(\partial^*\circ \iota(Z_1\wedge\dots\wedge Z_i\otimes v))\\
&=\sum_{k<j}(-1)^{k+j+1}Z_1\wedge\hat\dots\wedge Z_i \otimes d\tau([Z_k,Z_j]) v
\end{align*}
and
$$((\partial^*_0)^2\phi)_{j_1\dots j_i}=d\tau(\begin{pmatrix}
0& 0&-(i+1)(i+2)\\
0&0&0\\
0&0&0\\
\end{pmatrix})(\phi_{j_1\dots j_ia}{}^a).$$
holds for a $V$--valued form $\phi$. This means that $\partial^*_0$ is not a codifferential on the space of $V$--valued forms, but only on the space $L^i(V)$. We consider the $CSp(n,\mathbb{R})$--equivariant inclusion $$\iota_0: L^i(V)\to \wedge^i (\fl/\fp)^*\otimes V$$ and projections $$r_0: \wedge^i (\fl/\fp)^*\otimes V\to L^i(V),$$
where the projections $r_0$ are given by the $CSp(n,\mathbb{R})$--invariant complements of $Ker(d\tau(\fp_2))$ in $\wedge^i (\fl/\fp)^*\otimes V$ given by the trace components.

\begin{lem}
All maps $$\partial^*:=r_0\circ \partial^*_0\circ \iota_0=r_0\circ r\circ \partial^*\circ \iota \circ \iota_0$$ are codifferentials $L^{i+1}(V)\to L^{i}(V)$. There are codifferentials $\partial^*: L^{i+1}(V[2])\to L^{i}(V[2])$ constructed in the analogous way by using inclusion $\fp_2\otimes\wedge^i \fp_+\otimes W\subset \wedge^{i+1} \fp_+\otimes W$. 
\end{lem}

On $\wedge^i \fp_+\otimes W$, the differential $\partial$ is related to $\partial^*$ by dualizing $\fp_+\to \fp_+^*$ and then dualizing the formula using Killing form on $\frak{sp}(n+2,\mathbb{R})$. This duality  provides the formula
\begin{align*}
(\partial_0 \phi)&(X_1, \dots, X_{i+1}):=r(\sum_k (-1)^{k+1} d\tau(X_k) \phi(X_1,\hat\dots, X_{i+1})),
\end{align*}
where $X_k\in \fl_{-1}$. As in the case of $\partial^*_0$, this is not a differential on the space of all $V$--valued forms. The following is obtained by a computation analogous to computation of $(\partial^*_0)^2.$

\begin{lem}\label{5.2}
All maps $$\partial:=r_0\circ \partial_0\circ \iota_0=r_0\circ r\circ \partial\circ \iota \circ \iota_0$$ are differentials $L^{i}(V)\to L^{i+1}(V)$. There are differentials $\partial: L^{i}(V[2])\to L^{i+1}(V[2])$ constructed in the analogous way by using inclusion $\fp_2\otimes\wedge^i \fp_+\otimes W\subset \wedge^{i+1} \fp_+\otimes W$. 
\end{lem}

It is a simple consequence of the duality between $\partial^*$ and $\partial$ that $Im(\partial^*)$ is a complement of $Ker(\partial)$ and vice--versa. In particular, we can observe that analogously to the Hodge--decomposition $\wedge^i \fp_+\otimes W=Im(\partial^*)\oplus Ker(\partial^*)\cap Ker(\partial) \oplus Im(\partial)$, which is one of the ingredients of the construction in \cite{CSou}, we have decompositions $L^i(V)=Im(\partial^*)\oplus Ker(\partial^*)\cap Ker(\partial) \oplus Im(\partial)$ and $L^i(V[2])=Im(\partial^*)\oplus Ker(\partial^*)\cap Ker(\partial) \oplus Im(\partial)$. Let us emphasize that these three decompositions consist of different $CSp(n,\mathbb{R})$--modules and we write $\mathcal{H}^i(V),\mathcal{H}^i(V[2]),\mathcal{H}^i(W)$ for the spaces of sections of the cohomology bundles $\ba\times_P Ker(\partial^*)/Im(\partial^*)$. Let us emphasize that these are constructed as $\ba_0\times_{CSp(n,\mathbb{R})} Ker(\partial^*)/Im(\partial^*)$ in the splitting provided by the Weyl structure $\si^0$. We write just $\mathcal{H}^i$ for these three spaces in the case it is not necessary to distinguish between them.

\subsection{Running the BGG--machinery}

Now we can take the spaces of sections $\Omega^i$, $\mathcal{H}^i$ and the differentials and codifferentials $\partial, \partial^*$ we set up in the previous section and start running the BGG--machinery as in \cite{CSou}.

The final input for the construction in \cite{CSou} is a special class of differential operators between the bundles $\Omega^i$. We construct examples of such operators in the following section.

\begin{def*}
We call an operator $\mathcal{D}_i: \Omega^i\to \Omega^{i+1}$ compressable if it preserves the filtration of $\Omega^i$ and on the associated grading, the homogeneous part of degree zero of $\mathcal{D}_i$ coincides with $\partial$.
\end{def*}

As in \cite{CSou}, for the compressable operators $\mathcal{D}_i$, we find a unique splitting operator $\mathcal{L}_i(\mathcal{D}_i): \mathcal{H}^i\to \Omega^i$ with the properties
\begin{itemize}
\item $\partial^*\circ \mathcal{L}_i(\mathcal{D}_i)=0$,
\item $\pi\circ \mathcal{L}_i(\mathcal{D}_i)=\id$,
\item $\partial^*\circ \mathcal{D}_i\circ \mathcal{L}_i(\mathcal{D}_i)=0,$
\end{itemize}
where $\pi:Ker(\partial^*)\to Ker(\partial^*)/Im(\partial^*)$ is the natural projection. Let us show that as in \cite[Section 3]{CSou}, we can construct the splitting operator using the formula 
$$\mathcal{L}_i(\mathcal{D}_i)=\prod_j (\id-\frac{1}{a_{ji}}\partial^* \mathcal{D}_i)$$
where the integers $a_{ji}$ are determined as follows:

Consider section $\phi$ of $Ker(\partial^*)$, then $\partial^*\circ \mathcal{D}_i(\phi)$ is in filtration component that is one step higher than $\phi$. Then we can consider section $\phi'=\phi-\frac{1}{a_{ji}}\partial^* \mathcal{D}_i$ of $Ker(\partial^*)$ for unknown $a_{ji}$ and compute $\partial^*\circ \mathcal{D}_i(\phi')$. From definition of compressable operator, the part of $\partial^*\circ \mathcal{D}_i(\phi')$ of the lowest homogeneity is equal to $(\id-\frac{1}{a_{ji}}\partial^*\partial)\partial^*\circ \mathcal{D}_i(\phi')$. As in the case of parabolic geometries, $\partial^*\partial$ acts on each $CSp(n,\mathbb{R})$--submodule of $Im(\partial^*)$ by a single non--zero eigenvalue and thus by induction, all integers $a_{ij}$ are computed and the splitting operator constructed. In fact, from the example in the Appendix \ref{ApC} and the representation theory follows that all the eigenvalues of $\partial^*\partial$ on $Im(\partial^*)$ are negative.

Now, we can correctly define the BGG--like operators and BGG--like sequences.

\begin{def*}
For a compressable operators $\mathcal{D}_{i-1}$, we say that the $i$--th BGG--like operator induced by $\mathcal{D}_{i-1}$ is the operator $$\mathcal{B}_{i-1}:= \pi\circ \mathcal{D}_{i-1}\circ \mathcal{L}_{i-1}(\mathcal{D}_{i-1}).$$

For a sequence of compressable operators $\mathcal{D}_i$, we say that the sequence
$$
\xymatrix{
0\ar[r] & \mathcal{H}^0 \ar[r]^{\mathcal{B}_{0}}& \mathcal{H}^1\ar[r]^{\mathcal{B}_{1}}& \dots \ar[r]^{\mathcal{B}_{n-1}}& \mathcal{H}^n\ar[r]^{\mathcal{B}_{n}}& \mathcal{H}^{n+1}\ar[r]^{\mathcal{B}_{n+1}}&0\\
},
$$
is a BGG--like sequence.
\end{def*}

Finally, there is a version of \cite[Theorem 3.14]{CSou}, which in our setting is proven by the same proof.

\begin{thm}\label{1.8}
Suppose $\mathcal{D}_{i}\circ \mathcal{D}_{i-1}=0$ holds for a sequence of compressable operators for all $i$. Then the corresponding BGG--like sequence is a complex with the same cohomology.
\end{thm}

\section{Examples of compressable operators}\label{sec6}

Consider an ACAF--structure described by the Cartan geometry $(\ba\to M,\om)$ of type $(\fl,P,\Ad)$ with Weyl structure $\si^0$ such that $\nabla^0$ is the Weyl connection of $\si^0$ and the Rho--tensor $(\Rho_{ab},\Rho_a)$ is trace--free. The reason is that we want to consider also different normalizations of the Cartan geometries of type $(\fl,P,\Ad)$, because if we modify a compressable operator by a linear map of positive homogeneity, then the resulting operator remains compressable and part of this change can be seen as the change of the normalization.

At the beginning, we consider the same operators as in \cite[Section 4]{CSou}, the fundamental derivatives, and use them for the construction of compressable operators. We check when these compressable operators satisfy the conditions of the Theorem \ref{1.8}. We compute in the splitting provided by the Weyl structure $\si^0$ and in particular, we use the decomposition $s=s_-+s_0+s_1+s_2$ for a section $s\in \Gamma(\mathcal{A})$ given by components in $\ba_0\times_{CSp(n,\mathbb{R})}(\fl_{-1}\oplus \frak{csp}(n,\mathbb{R})\oplus \fp_1\oplus \fp_2)$. Further, we use  the induced (Weyl) connection $\nabla^0$ on the space of sections of $\ba_0\times_{CSp(n,\mathbb{R})}U$ for any $CSp(n,\mathbb{R})$--module $U$ and denote $\Rho: \fl_{-1}^*\to (\fp_1\oplus \fp_2)$ for the map given by $\Rho_{ab}s_-^a+\Rho_{a}s_-^a$. We consider a tractor--like bundle $$\mathcal{V}:=\ba\times_P V$$ for a restriction of representation $\tau: Sp(n+2,\mathbb{R})\to GL(W)$ to a $P$--invariant subspace $V\subset W$.

We define an adjoint tractor--like bundle $$\mathcal{A}:=\ba\times_P \fl$$ for the representation $\tau=\Ad: Sp(n+2,\mathbb{R})\to GL(\frak{sp}(n+2,\mathbb{R}))$. We denote by $\Gamma(\mathcal{A})$ the space of sections of $\mathcal{A}$. Each section $s\in \Gamma(\mathcal{A})$ corresponds to $P$--invariant vector field $\om^{-1}(s)$ on $\ba$. We define \emph{the fundamental derivative} in the usual way as an operator 
$$d^\om: \Omega^0\to \mathcal{A}^*\otimes \Omega^0$$ given by formula $$d^\om(s)v:=\om(\om^{-1}(s).v)$$ for each section $v\in \Omega^0$ and section $s\in \mathcal{A}$, where $.$ is the usual directional derivative of the function $v: \ba\to V$ in direction of the vector field $\om^{-1}(s)$.
We extend the fundamental derivative by iterating and antisymmetrizing to the operators
$$d^\om: \wedge^{k-1}\mathcal{A}^*\otimes\Omega^0\to\wedge^{k} \mathcal{A}^*\otimes \Omega^0.$$
Let us show that the fundamental derivative has the usual properties known from \cite[Section 1.5]{parabook}, \cite[Section 4]{CSou}.

\begin{lem}\label{l6.1}
It holds
$$(d^\om)^2=Alt_2(d^\om(-d\om)),$$
where $Alt_2$ is the antisymmetrization operator $\wedge^{2}\mathcal{A}^*\otimes \wedge^{k}\mathcal{A}^*\otimes\Omega^0\to\wedge^{k+2} \mathcal{A}^*\otimes \Omega^0$ and $d\om$ is the section of $\wedge^{2}\mathcal{A}^*\otimes \mathcal{A}$ corresponding (via $\om^{-1}$) to the differential of $\om$.

For $v\in \Omega^0$ and $s,t\in \Gamma(\mathcal{A})$, it holds
$$d^\om(s)v=\nabla^0_{s_-}v+d\tau(\Rho(s_-))(v)-d\tau(s_0+s_1+s_2)(v),$$
and
$$d\om(s,t)=\ad(s_0+s_1+s_2)(t)-\ad(t_0+t_1+t_2)(s)+\tilde R^\Rho(s_-,t_-),$$
where
\begin{align*}
\tilde R^\Rho_{ab}&=
\begin{pmatrix}
\Rho_{ab}-\Rho_{ba}&\Rho_aJ_{bc}-\Rho_bJ_{ac}&2\Rho_{af}\Rho_b{}^f\\
0&\delta_a{}^d\Rho_{bc}-\delta_b{}^d\Rho_{ac}-J_{bc}\Rho_{a}{}^d+J_{ac}\Rho_{b}{}^d&\Rho_b\delta_a{}^d-\Rho_a\delta_b{}^d\\
0&0&-\Rho_{ab}+\Rho_{ba}\\
\end{pmatrix}\\
&+\begin{pmatrix}
-\frac1n(R^0)_{abi}{}^{i}& \nabla^0_a\Rho_{bc}-\nabla^0_b\Rho_{ac}&\nabla^0_a\Rho_b-\nabla^0_b\Rho_a\\
2(H+S)_{ab}{}^d&(R^0)_{abc}{}^{d}-\frac1n(R^0)_{abi}{}^{i}\delta_c{}^d&\nabla^0_a\Rho_{b}{}^d-\nabla^0_b\Rho_{a}{}^d\\
0&-2(H+S)_{abc}&\frac1n(R^0)_{abi}{}^{i}\\
\end{pmatrix}\\
\end{align*}
is a $\fl$--valued two--form given by the torsions $H,S$ and the curvature $R^0$ of $\nabla^0$ and by the Rho--tensors $\Rho_{ab},\Rho_a$.
\end{lem}
\begin{proof}
The first formula follows as in the proof of \cite[Theorem 4.3]{CSou}. The result formula looks different only because we need to express the Ricci identity from \cite[Proposition 5.9]{parabook} using $d\om$ instead of the curvature (which is not well--defined in our case).

We obtain the second formula in the same way as in \cite[Proposition 5.1.10]{parabook}. 

Finally, 
$$d\om(s,t)=\ad(s_0+s_1+s_2)(t)-\ad(t_0+t_1+t_2)(s)+d\om(s_-,t_-)$$
follows from $P$--equivariance of $\om$. The definition of Rho--tensors implies that 
$$(\si^0)^*\om=(\id+\Rho)\circ \om^0$$
 and thus 
 $$\om^{-1}(s_-)+\om^{-1}(\Rho(s_-))=T\si^0(\om^0)^{-1}(s_-)$$ on $\si^0(\ba_0)$. Therefore 
\begin{align*}
&d\om(\om(T\si^0(\om^0)^{-1}(s_-)),\om(T\si^0(\om^0)^{-1}(t_-)))=\\
&d\om(s_-,t_-)-\ad(\Rho(s_-))(t_-)+\ad(\Rho(t_-))(s_-)-\ad(\Rho(s_-))(\Rho(t_-))
\end{align*}
  and 
  $$\ad(\Rho(s_-))(t_-)-\ad(\Rho(t_-))(s_-)+\ad(\Rho(s_-))(\Rho(t_-))$$ is the first matrix in the formula for $\tilde R^\Rho_{ab}$. Now, the differential commutes with the pullback by $\si^0$ and thus 
 \begin{align*}
d\om(\om(T\si^0(\om^0)^{-1}(s_-)),&\om(T\si^0(\om^0)^{-1}(t_-)))=d((\id+\Rho)\circ \om^0)(s_-,t_-)\\
&=d\Rho(s_-)(t_-)-d\Rho(t_-)(s_-)+(\id+\Rho)\circ d\om^0(s_-,t_-).
\end{align*} Now, the second matrix in the formula for $\tilde R^\Rho_{ab}$ is 
$$(\nabla^0_{s_-}\Rho)(t_-)-(\nabla^0_{t_-}\Rho)(s_-)+d\om^0(s_-,t_-),$$
because if we rewrite $d\Rho$ using the connection $\nabla^0$, then the action of the torsion on $\Rho$ cancels with $\Rho\circ d\om^0(s_-,t_-)$. It is clear that $d\om^0(s_-,t_-)$ is the curvature of $\om^0$, which contains both torsions $H,S$ in $\fl_{-1}$--slot and the curvature $R^0$ in $\frak{csp}(\mathbb{R})$--slot.
\end{proof}

Now, we are ready to define first examples of the compressable operators

\subsection{Twisted exterior derivatives}

On $\Omega^i(W)$, we extend the fundamental derivative $d^\om$ to sections of $\ba\times_P \frak{sp}(n+2,\mathbb{R})$ by defining $$d^\om(s_{-2})v:=0$$ for $s_{-2}\in \fp_2^*$ in the splitting given by the Weyl structure $\si^0$. Then a consequence of Lemma \ref{l6.1} is that \emph{the twisted exterior derivative} $d^W:  \Omega^i(W)\to \Omega^{i+1}(W)$ is well--defined  by formula
$$d^W:=d^\om+\partial_\frak{sp},$$
where $\partial_\frak{sp}$ is the usual Lie algebra differential on $\wedge^i \frak{sp}(n+2,\mathbb{R})^*\otimes W\to \wedge^{i+1} \frak{sp}(n+2,\mathbb{R})^*\otimes W$. Indeed, 
\begin{align*}
(d^W\phi)((X_1+x_1)&,\dots,(X_{i+1}+x_{i+1}))=\\
&\sum_k (-1)^k(\nabla^0_{X_k}\phi+d\tau(\Rho(X_k))\phi)((X_1+x_1),\hat\dots,(X_{i+1}+x_{i+1}))\\
&+(\partial \phi)((X_1+x_1),\dots,(X_{i+1}+x_{i+1}))
\end{align*}
for $X_k\in \fl_{-1},x_k\in \fp_2^*$.

On $\Omega^i(L^i(V)),\Omega^i(L^i(V[2]))$, we use the inclusions $\iota,\iota_0$ and projections $r,r_0$ and define \emph{a twisted exterior derivative} as
\begin{align*}
d^V\phi:=r_0\circ r\circ d^W\circ\iota\circ \iota_0(\phi).
\end{align*}

By definition, the twisted exterior derivatives are comprassible operators and define the corresponding BGG--like sequences. Thus it remains to compute $(d^V)^2$ and decide when $(d^V)^2=0$.

\begin{prop}\label{prop6.2}
It holds 
$$(d^W)^2=Alt_2(d^\om(-\tilde R^\Rho)),$$
where $Alt_2$ is the antisymmetrization operator $\ba_0\times_{CSp(n,\mathbb{R})}\wedge^2(\fl/\fp)^*\otimes \Omega^i \to \Omega^{i+2}$. Moreover, if the Rho--tensor vanishes, then also
$$(d^V)^2=Alt_2(r_0\circ r\circ d^\om(-\tilde R^\Rho)\circ \iota \circ \iota_0).$$

In particular, if the ACAF--structure is flat, i.e., $\tilde R^\Rho=0$, and the Rho--tensor vanishes in the case $V\neq W$, then $(d^W)^2=0$ and  $(d^V)^2=0$ and the corresponding BGG--like sequences are complexes.
\end{prop}
\begin{proof}
If Rho--tensor vanishes, then $d^W(Ker(r_0\circ r))\subset Ker(r_0\circ r)$ and thus 
$$(d^V)^2=r_0\circ r\circ (d^W)^2\circ \iota \circ \iota_0.$$ The proof of \cite[Lemma 4.2]{CSou} shows that $(d^\om\circ\partial_\frak{sp}+\partial_\frak{sp}\circ d^\om)(s_-,t_-)=0$ for $s,t\in \Gamma(\mathcal{A})$ and thus $(d^W)^2=(d^\om)^2=Alt_2(d^\om(-\tilde R^\Rho))$ follows from Lemma \ref{l6.1}. The remaining claims follow from the Theorem \ref{1.8}.
\end{proof}

\subsection{Examples of compressable operators on $\Omega^i(W)$}\label{sec5.2}

The BGG--like sequence on $\Omega^i(W)$ for the twisted exterior derivative is by construction a sequence of invariant differential operators with the same symbol as the descended BGG--sequence in \cite{CSaII} for the same representation $\tau$ on $W$. Therefore on CF--structures the difference between our BGG--like operators and the operators from  \cite{CSaII} is an invariant differential operator of lower order.

Let us now discuss which modification of the twisted exterior derivative provides a compressable operator such that the corresponding BGG--like sequences is a complex that coincides with BGG--complexes on CF--structures with curvature of Ricci type in \cite{CSaII,ES}. Let us recall that CF--structures with curvature of Ricci type means that $H=S=0$ and $W_{abcd}=0$ holds for the essential curvature components from Appendix \ref{ApB}. Under this assumption, the curvature $R^0$ of $\nabla^0$ is of the form,
\begin{align*}
(R^0)_{abcd}&=\Theta_{ac}J_{bd}-\Theta_{bc}J_{ad}+\Theta_{ad}J_{bc}-\Theta_{bd}J_{ac}-2\Theta_{cd}J_{ab},
\end{align*}
where $\Theta_{ab}$ is a section of $S^2T^*M$ corresponding to the remaining essential curvature component, see Theorem \ref{escur}.

We can use some results from \cite{ES} to obtain this modification. Namely, we consider different normalization of the Cartan geometry of type $(\fl,P,\Ad)$ and assume that the equalities $$\Rho_{ab}=\Theta_{ab},\Rho_a=\frac{1}{n+1}(\nabla^0)^b\Theta_{ab}$$
holds for the Rho--tensors of the Weyl structure $\si^0$. Then the Bianchi identity 
\begin{align*}
0=&J^{de}\nabla_{[e}R^0_{ab]cd}\\
=&J_{ac}(\nabla^0)^d\Theta_{bd}-J_{bc}(\nabla^0)^d\Theta_{ad}+2J_{ab}(\nabla^0)^d\Theta_{cd}+(n+1)(\nabla^0_a\Theta_{bc}-\nabla^0_b\Theta_{ac})
\end{align*} implies that 
$$\nabla^0_a\Rho_{bc}-\nabla^0_b\Rho_{ac}+\Rho_aJ_{bc}-\Rho_bJ_{ac}=-J_{ab}\frac{2}{n+1}\nabla^d\Theta_{cd}.$$ Further, as in \cite[Lemma 4]{ES}, we obtain $$\nabla^0_a\Rho_b-\nabla^0_b\Rho_a+2\Rho_{af}\Rho_b{}^f=-\frac{2}{n}J_{ab}(\frac{1}{n+1}(\nabla^0)^c(\nabla^0)^d\Theta_{cd}+\Theta_{ef}\Theta^{ef})$$ and thus
\begin{align*}
\tilde R^\Rho_{ab}&=
-2J_{ab}\begin{pmatrix}
0& \frac{1}{n+1}\nabla^e\Theta_{ce}&\frac1n(\frac{1}{n+1}(\nabla^0)^c(\nabla^0)^d\Theta_{cd}+\Theta_{ef}\Theta^{ef})\\
0&\Theta_{c}{}^{d}& \frac{1}{n+1}\nabla^e\Theta_{e}{}^d\\
0&0&0\\
\end{pmatrix}.
\end{align*}
Let us observe that we can decompose $d^W=d^{\nabla^W}+\partial_2$, where $d^{\nabla^W}$ is the covariant differential $d^{\nabla^W}$ of the connection $$\nabla^W:=\nabla^0+d\tau(\Rho)+(\partial-\partial_2)$$ and $\partial_2: \wedge^i \fp_1\otimes W\to \wedge^i \fp_1\otimes W[2]$ is defined as $$(\partial_2\phi)(s_{-2})=s_{-2}d\tau(\begin{pmatrix}
0& 0&0\\
0&0&0\\
1&0&0\\
\end{pmatrix})\phi$$ for $s_{-2}\in \fp_2^*$. We can directly observe from Lemma \ref{a2} that the curvature of $\nabla^W$  is of the form $-2J_{ab}\Theta$ for section $\Theta$  of $ \ba\times_P\frak{sp}(n+2,\mathbb{R})\otimes \mathbb{R}[2]$. Therefore $\tilde R^\Rho_{ab}=-2J_{ab}(\Theta-\partial_2)$.

As in \cite[Lemma 9]{ES}, let us consider maps $d\tau(\Theta): \wedge^i \fp_1\otimes W\to \wedge^i \fp_1\otimes W[2]$ induced by application of $\Theta$ on $W$ using $d\tau$ and $J\wedge:  \wedge^i \fp_1\otimes W[2]\to \wedge^{i+2} \fp_1\otimes W$ induced by antisymmetrization of $J\otimes \phi$. Then we can define 
$$\mathcal{D}_i:=d^W+d\tau(\Theta)+2J\wedge=d^{\nabla^W}+d\tau(\Theta)+2J\wedge.$$ 
These are exactly the compressable operators from \cite[Lemma 9]{ES} and $\mathcal{D}_i\circ \mathcal{D}_{i-1}=0$ holds. Consequently, from the Theorem \ref{1.8} follows that the corresponding BGG--like sequences are complexes that coincide with the complexes in \cite{ES,CSaII}.

We can directly compute $\mathcal{D}_i\circ \mathcal{D}_{i-1}=Alt_2(d\tau(-2J\Theta))+2J\wedge d\tau(\Theta)+d\tau(\Theta)2J\wedge+d^{\nabla^W}\circ(d\tau(\Theta)+2J\wedge)+(d\tau(\Theta)+2J\wedge)\circ d^{\nabla^W}=d^{\nabla^W}\circ(d\tau(\Theta)+2J\wedge)+(d\tau(\Theta)+2J\wedge)\circ d^{\nabla^W}=0$ using the Bianchi identity $d^{\nabla^W}\Theta=0$ and the identity $d^{\nabla^W}J=0$.

\appendix
\section{Cartan geometries modeled on skeletons}\label{ApS}

In this article, we consider more general models of Cartan geometries than the Klein geometries. Let us recall the definitions in this situation.

\begin{def*}\label{defskel}
We say that the triple $(\fg,H,\Ad)$ is a \emph{skeleton} if

\begin{enumerate}
\item $H$ is a Lie group with Lie algebra $\fh$,
\item $\fg$ is a vector space with vector subspace $\fh$, and
\item $\Ad$ is a representation of $H$ on $\fg$ such that $\Ad|_\fh$ is the adjoint representation of $H$ on $\fh$.
\end{enumerate}

We say that the pair $(\ba\to M,\om)$ is a \emph{Cartan geometry of type $(\fg,H,\Ad)$} if

\begin{enumerate}
\item $\ba\to M$ is a principal $H$--bundle over the smooth connected manifold $M$,
\item $\om$ is a $\fg$--valued one--form on $\ba$ such that $$(r^h)^*\om=\Ad(h)^{-1}\circ \om$$ holds for all $h\in H$,  $\om(\zeta_X)=X$ holds for all $X\in \fh$ and $\om(u): T_u\ba\to \fg$ is a linear isomorphism for all $u\in \ba$, where we denote by $r^h$ the right action of $h\in H$ on $\ba$ and by $\zeta_X$ the fundamental vector field of $X\in \fh$ on $\ba$
\end{enumerate}

A morphism of Cartan geometries of type $(\fg,H,\Ad)$ between $(\ba\to M,\om)$  and $(\tilde \ba\to \tilde M,\tilde \om)$   is a principal $H$--bundle morphism $\phi: \ba\to \tilde \ba$ such that $\phi^*(\tilde \om)=\om$.
\end{def*}

Many of the properties of the Cartan geometries depend only on the skeleton, see \cite{ja-arx}. In particular, this is the case for the relation between the morphisms of Cartan geometries and the underlying (local) diffeomorphisms of the base manifolds.

\begin{prop}\label{rigid}
Let $(\fg,H,\Ad)$ be skeleton. Suppose the maximal normal Lie subgroup $N$ of $H$ with the property that $\Ad(n)\fg\subset \frak{n}$ holds for all $n\in N$ is trivial. Then if $\phi_1$, $\phi_2$ are two morphisms of Cartan geometries of type $(\fg,H,\Ad)$ covering the same diffeomorphism of the base manifolds, then $\phi_1=\phi_2$.
\end{prop}
\begin{proof}
The proposition is essentially proved by the proof for the analogous statement for Cartan geometries modeled on Klein geometries, c.f. \cite[Theorem 1.5.3]{parabook}, and the further details are in \cite[Section 1]{ja-arx}.
\end{proof}

\section{Curvature of the totally trace--free ACS--connections}\label{ApB}

Let us consider an ACS--connection $\nabla^0$ with totally trace--free torsion and the corresponding torsion free--connection $D^0$. We relate the curvature $\kappa^0$ of the connection $D^0$ with the curvature $R^0$ of the connection $\nabla^0$. Let us recall that $D^0$ in the projective structure is determined by the condition $D^0_aJ_{bc}=-H_{bca}+2S_{bca}$. In the CF--case when $H=S=0$ and $D^0=\nabla^0$, our formulas coincide with the formulas in \cite{ES} up to a sign convention for the projective Rho--tensor and the convention for $n$.

If we view the curvature $(\kappa^0)_{abc}{}^{d}$ of $D^0$ as the section of $\ba_0\times_{CSp(n,\mathbb{R})}\wedge^2(\mathbb{R}^n)^*\otimes \frak{gl}(n,\mathbb{R})$ that is satisfying the Bianchi identity $(\kappa^0)_{[abc]}{}^{d}=0$, then $(\kappa^0)_{abcd}$ decomposes into $CSp(n,\mathbb{R})$--irreducible components as follows:
\begin{align*}
(\kappa^0)_{abcd}&=W_{abcd}+Y_{abcd}+Z_{abcd}\\
&-\frac{3}{n-1}\Theta_{ac}J_{bd}+\frac{3}{n-1}\Theta_{bc}J_{ad}+\Theta_{ad}J_{bc}-\Theta_{bd}J_{ac}-2\Theta_{cd}J_{ab}\\
&+\frac{2}{n+1}\Sigma_{ab}J_{cd}+\frac{1}{n+1}\Sigma_{ac}J_{bd}-\frac{1}{n+1}\Sigma_{bc}J_{ad}+\Sigma_{ad}J_{bc}-\Sigma_{bd}J_{ac}-2\Sigma_{cd}J_{ab}\\
&-(\Rho_S)_{bc}J_{ad}+(\Rho_S)_{ac}J_{bd}-(\Rho_A)_{bc}J_{ad}+(\Rho_A)_{ac}J_{bd}+2(\Rho_A)_{ab}J_{cd}\\
&-\Rho_TJ_{bc}J_{ad}+\Rho_TJ_{ac}J_{bd}+2\Rho_TJ_{ab}J_{cd},
\end{align*}
where
\begin{align*}
&W_{[abc]d}=0, W_{ab[cd]}=0, W_{abc}{}^c=0,W_{a}{}^a{}_{cd}=0,\\
&Y_{[abc]d}=0, Y_{ab(cd)}=0,Y_{abc}{}^c=0,Y_{abcd}=Y_{cdab},\\
&Z_{[abc]d}=0, Z_{ab(cd)}=0,Z_{abc}{}^c=0,Z_{abcd}=-Z_{cdab},\\
&\Theta_{[ab]}=0,\Sigma_{(ab)}=0,\Sigma_{a}{}^{a}=0,(\Rho_S)_{[ab]}=0,(\Rho_A)_{(ab)}=0,(\Rho_A)_{a}{}^{a}=0.
\end{align*}
The first three rows form the projective Weyl tensor and depend only on the projective structure $[D]$. Further, $(\Rho_S)_{ab}+(\Rho_A)_{ab}+\Rho_TJ_{ab}$ is the projective Rho--tensor of the connection $D^0$.

Similarly, if we view the curvature $(R^0)_{abc}{}^{d}$ of $\nabla^0$ as the section of $\ba_0\times_{CSp(n,\mathbb{R})}\wedge^2(\mathbb{R}^n)^*\otimes \frak{csp}(n,\mathbb{R})$, then we can decompose it to irreducible components as follows:
\begin{align*}
(R^0)_{abcd}&=U_{abcd}+V_{abcd}+V_{abdc}+2A_{cd}J_{ab}+2B_{ab}J_{cd}+FJ_{ab}J_{cd}\\
&+C_{ac}J_{bd}-C_{bc}J_{ad}+C_{ad}J_{bc}-C_{bd}J_{ac}\\
&+E_{ac}J_{bd}-E_{bc}J_{ad}+E_{ad}J_{bc}-E_{bd}J_{ac}
\end{align*}
where
\begin{align*}
&A_{[ab]}=0, B_{(ab)}=0, B_{a}{}^a=0, C_{[ab]}=0,E_{(ab)}=0, E_{a}{}^a=0,\\
& U_{[abc]d}=0, U_{ab[cd]}=0, U_{a}{}^a{}_{cd}=0,\\
&V_{a(bc)d}=0,V_{[abcd]}=0, V_{a}{}^a{}_{cd}=0.
\end{align*}

The difference of the curvatures $\kappa^0$ and $R^0$ is completely determined by the torsions $H$ and $S$ and there is the following formula for the difference
\begin{align*}
(R^0)_{abcd}&=(\kappa^0)_{abcd}+K_{abcd}\\
K_{abcd}&:=\nabla^0_a(H+S)_{bcd}-\nabla^0_b(H+S)_{acd}\\
&+2(H+S)_{ced}(H+S)_{ba}{}^e-(H+S)_{aed}(H+S)_{bc}{}^{e}+(H+S)_{bed}(H+S)_{ac}{}^e.
\end{align*}
We can compare the components on the both sides and conclude:

\begin{thm}\label{escur}
The tensors $\kappa^0$ and $R^0$ are determined by torsions $H,S$ and the curvature components $W,\Theta$. Moreover, the tensors
\begin{align*}
(R^0)_{aic}{}^i&=(n-1)(\Rho_S)_{ac}+(n+1)(\Rho_A)_{ac}-\nabla^0_e(H+S)_{ac}{}^e
-(H+S)_{ce}{}^f(H+S)_{af}{}^e\\
(R^0)_{abi}{}^i&=2(n+1)(\Rho_A)_{ab}
\end{align*}
are trace--free.
\end{thm}
\begin{proof}
Taking the trace twice and considering the trace--freeness of the torsion we obtain
$$
F=\Rho_T =0
$$

Symmetrizing and antisymmetrizing the last two entries we obtain that
$$(R^0)_{abcd}=(R^0)_{ab(cd)}+2B_{ab}J_{cd}$$
and
$$2B_{ab}J_{cd}= (\kappa^0)_{ab[cd]}+K_{ab[cd]}.$$
Therefore the components $Y,Z$ can be determined by $H,S$ and the trace type components.

Further, the Bianchi identity allows us to express the remaining components that are not of the trace type in the following way
\begin{align*}
V_{abcd}+V_{abdc}
&=\frac14(K_{abcd}+K_{bcad}+K_{cabd}+K_{abdc}+K_{bdac}+K_{dabc}-4(A_{cd}+C_{cd})J_{ab}\\
&-2(A_{ad}+C_{ad}-B_{bd}+2E_{ad})J_{bc}-2(-A_{bc}-C_{bc}+B_{bc}-2E_{bc})J_{ad}\\
&-2(A_{ac}+C_{ac}-B_{ac}+2E_{ac})J_{bd}-2(-A_{bd}-C_{bd}+B_{bd}-2E_{bd})J_{ac})
\end{align*}
and
\begin{align*}
&U_{abcd}\\
&=W_{abcd}+(C_{cd}-2\Theta_{cd}-A_{cd})J_{ab}\\
&+(\frac{n-4}{2(n-1)}\Theta_{ac}-\frac12C_{ac}+\frac12(\Rho_S)_{ac}+\frac{n+2}{2(n+1)}\Sigma_{ac}+\frac12(\Rho_A)_{ac}+\frac12(A_{ac}-B_{ac}))J_{bd}\\
&+(\frac{n-4}{2(n-1)}\Theta_{ad}-\frac12C_{ad}+\frac12(\Rho_S)_{ad}+\frac{n+2}{2(n+1)}\Sigma_{ad}+\frac12(\Rho_A)_{ad}+\frac12(A_{ad}-B_{ad}))J_{bc}\\
&-(\frac{n-4}{2(n-1)}\Theta_{bc}-\frac12C_{bc}+\frac12(\Rho_S)_{bc}+\frac{n+2}{2(n+1)}\Sigma_{bc}+\frac12(\Rho_A)_{bc}+\frac12(A_{bc}-B_{bc}))J_{ad}\\
&-(\frac{n-4}{2(n-1)}\Theta_{bd}-\frac12C_{bd}+\frac12(\Rho_S)_{bd}+\frac{n+2}{2(n+1)}\Sigma_{bd}+\frac12(\Rho_A)_{bd}+\frac12(A_{bd}-B_{bd}))J_{ac}\\
&+\frac14(K_{abcd}-K_{bcad}-K_{cabd}+K_{abdc}-K_{bdac}-K_{dabc}).
\end{align*}

Computing all the possible traces and solving the resulting equations gives
\begin{align*}
A^0_{ab}&=-\Theta_{ab}+\frac{(nK_e{}^e{}_{ba}+nK_e{}^e{}_{ab}-2K_{eb}{}^e{}_a-2K_{ea}{}^e{}_b-2K_{eba}{}^e-2K_{eab}{}^e)}{4(n-2)(n+2)}\\
B^0_{ab}&=\frac{(2(-K_{eba}{}^e+K_{eab}{}^e)+K_{e}{}^e{}_{ab}-K_{e}{}^e{}_{ba}+2K_{eb}{}^e{}_a-2K_{ea}{}^e{}_b)}{4(n-4)}\\
C^0_{ab}&=\Theta_{ab}+\frac{(-(n+1)(K_{eb}{}^e{}_a+K_{ea}{}^e{}_b)
+K_e{}^e{}_{ba}+K_e{}^e{}_{ab}+K_{eba}{}^e+K_{eab}{}^e)}{2(n-2)(n+2)}\\
E^0_{ab}&=\frac{((n-2)(K_{e}{}^e{}_{ab}-K_{e}{}^e{}_{ba})+2(n-2)(K_{eb}{}^e{}_a-K_{ea}{}^e{}_b)-4K_{eba}{}^e+4K_{eab}{}^e)}{4(n-4)n}\\
\Sigma_{ab}&=\frac{(
(n^2-2n-4)(K_{e}{}^e{}_{ab}-K_{e}{}^e{}_{ba})+2n(K_{eb}{}^e{}_a-K_{ea}{}^e{}_b)+2n(-K_{eba}{}^e+K_{eab}{}^e))}{4(n-4)n(n+2)}\\
(\Rho_A^0)_{ab}&=\frac{n(K_{e}{}^e{}_{ab}-K_{e}{}^e{}_{ba}+2K_{eb}{}^e{}_a-2K_{ea}{}^e{}_b-2K_{eba}{}^e+2K_{eab}{}^e)}{4(n-4)(n+1)}\\
(\Rho_S^0)_{ab}&=\frac{(n+2)}{(n-1)}\Theta_{ab}+\frac{(K_{eba}{}^e+K_{eab}{}^e-K_{eb}{}^e{}_a-K_{ea}{}^e{}_b)}{2(n-2)}
\end{align*}
and the claim of the theorem is proved.
\end{proof}

\section{Tractor--like connections and $\partial^*$--normal Cartan geometries of type $(\fl,P,\Ad)$}\label{ApA}

It is simple observation that the twistor exterior derivative $$d^W: \Omega^0\to \Omega^1(L^i(W))$$
is in fact a linear connection $\nabla^W$, which we can call \emph{a tractor--like connection.} 

\begin{lem}\label{a2}
Let $W$ be a representation of $Sp(n+2,\mathbb{R})$. Then the curvature of $\nabla^W$ is $$d\tau(\tilde R^\Rho_{ab})+2J_{ab}\partial_2$$
and the projection of the curvature of $\nabla^W$ onto $L^2(W)$ is $d\tau(\tilde R^\Rho_{ab})$, where $\tilde R^\Rho_{ab}$ is defined in Lemma \ref{l6.1} and $\partial_2$ is defined in Section \ref{sec5.2}.
\end{lem}
\begin{proof}
The curvature of $\nabla^W$ is given as $d^{\nabla^W}\circ \nabla^W$. The difference between $d^{\nabla^W}\circ \nabla^W$ and $(d^W)^2$ is $d^W(2(H+S))+2J_{ab}\partial_2$. Now, $(d^W)^2+d^W(2(H+S))=d\tau((\tilde R^\Rho))$ follows from  the Lemma \ref{l6.1}, because $d^\om(2(H+S)-(\tilde R^\Rho))+(\partial_\frak{sp})(2(H+S))=d\tau((\tilde R^\Rho)-2(H+S))+d\tau(2(H+S))$.
\end{proof}

Let us recall that the parabolic geometries are usually normalized by the condition that the curvature of the parabolic geometry is in the kernel of $\partial^*$. For the Cartan geometries modeled on skeletons, the curvature can not be defined, in general. However, we can try to choose a normalization that $\partial^* \tilde R_{ab}^\Rho=0$ holds for the Rho--tensor $\Rho$ of $\si^0$, where
\begin{align*}
\frac12(\partial^*\tilde R^\Rho)_{a}&=
\begin{pmatrix}
0&\Rho_{ca}-(n+1)\Rho_{ac}&2(1-n)\Rho_a-2\nabla^0_i\Rho_{a}{}^i\\
0&0&\Rho^d{}_{a}-(n+1)\Rho_{a}{}^d\\
0&0&0\\
\end{pmatrix}\\
&+\begin{pmatrix}
0&(R^0)_{aic}{}^{i}&0\\
0&-2H_{ac}{}^d-2H_{a}{}^d{}_c&(R^0)_{ai}{}^{di}\\
0&0&0\\
\end{pmatrix}.
\end{align*}

\begin{prop}
There is a Cartan geometry of type $(\fl,P,\Ad)$ describing the ACAF--structure such that $\partial^* \tilde R_{ab}=0$  if and only if $H=0$, i.e., if it is an ACF--structure. The unique Rho--tensors describing this normalization are
\begin{align*}
\Rho_{ac}&=\frac{n+1}{n(n+2)}(R^0)_{aic}{}^{i}+\frac{1}{n(n+2)}(R^0)_{cia}{}^{i}\\
&=\frac{n-1}{n}(\Rho_S)_{ac}+\frac{n+1}{n+2}(\Rho_A)_{ac}-\frac{1}{n+2}\nabla^0_eS_{ac}{}^e-\frac{1}{n}S_{ce}{}^fS_{af}{}^e\\
\Rho_a&=\frac{1}{1-n}\nabla^0_i\Rho_{a}{}^i
\end{align*}
\end{prop}

Let us emphasize that for such $\partial^*$--normal Cartan geometries the curvature $\tilde R^\Rho_{ab}$ can be obtained using the splitting operator from the so--called harmonic curvature sitting in the cohomology $\mathcal{H}^2(L^i(\frak{sp}(n+2,\mathbb{R})))$. The harmonic components of the curvature can be identified with $S$ and the components of $R^0$ given by $U$ and the subspace of the tensors generated by $A,C$ satisfying $A_{ac}=\frac{n}{2}C_{ac}$.

\section{BGG--like sequence on standard tractor bundle}\label{ApC}

We explicitly construct the BGG--like sequence for tractor bundle given by the standard representation $W=T$ of $Sp(n+2,\mathbb{R})$. The representation $T$ decomposes as $CSp(n,\mathbb{R})$--module to $\mathbb{R}[1]\oplus (\mathbb{R}^{n})^*[-1]\oplus \mathbb{R}[-1]$. Therefore we write the sections of $\ba_0\times_{CSp(n,\mathbb{R})} \wedge^k(\fl/\fp)^*\otimes T$ in the splitting provided by the distinguished Weyl structure $\si^0$ of the ACAF--structure $(\ba\to M,\om)$ as the columns
$$\begin{pmatrix}
r_{a_1\dots a_k}\\
s_{a_1\dots a_kd}\\
t_{a_1\dots a_k}\\
\end{pmatrix},$$
where the forms are antisymmetric in the $a_i$ indices and $\mathbb{R}[-1]$ is the top slot (this is opposite convention than in \cite{ES}).

The maps $\partial^*_0$ and $(\partial^*_0)^2$ on $\ba_0\times_{CSp(n,\mathbb{R})} \wedge^k(\fl/\fp)^*\otimes T$ have the following formulas
\begin{align*}
\partial^*_0\begin{pmatrix}
r_{a_1\dots a_k}\\
s_{a_1\dots a_kd}\\
t_{a_1\dots a_k}\\
\end{pmatrix}&=(-1)^kk\begin{pmatrix}
s_{a_1\dots a_{k-1}i}{}^{i}\\
t_{a_1\dots a_{k-1}d}\\
0\\
\end{pmatrix},\\
(\partial^*_0)^2\begin{pmatrix}
r_{a_1\dots a_k}\\
s_{a_1\dots a_kd}\\
t_{a_1\dots a_k}\\
\end{pmatrix}&=-k(k+1)\begin{pmatrix}
t_{a_1\dots a_{k-2}i}{}^{i}\\
0\\
0\\
\end{pmatrix}.
\end{align*}
We see that the subbundles $L^i(T)$ on which $(\partial^*_0)^2=0$ are exactly characterized by the property that the tensors $t_{a_1\dots a_k}$ are totally trace-free. In particular, the map $r_0$ defining $\partial^*$ on $L^i(T)$ is the projection on the totally trace--free part of the last slot.

It is fairly simple to compute the cohomology $Ker(\partial^*)/Im(\partial^*)$ for fixed $n$ just by decomposing the $CSp(n,\mathbb{R})$--representations in $L^i(T)$. For example for $n=6$, we obtain the following diagram characterizing the cohomologies $\mathcal{H}^i(L^i(T)):$
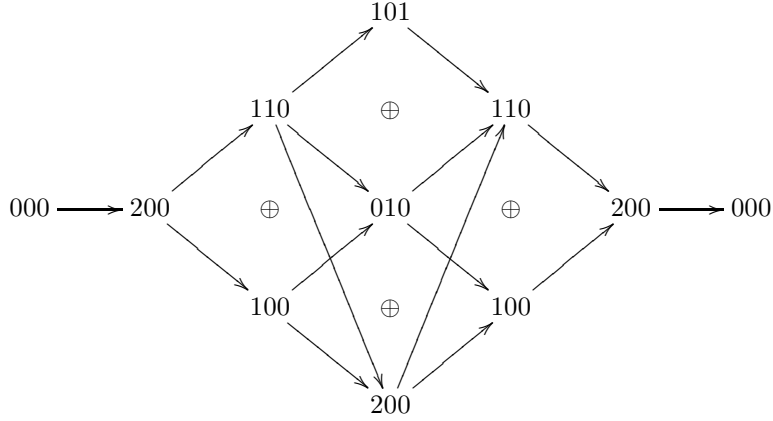
$$
\xymatrix{
&&&101\ar[rd]&&&\\
&&110\ar[ru] \ar[rddd]\ar[rd]&\oplus&110\ar[rd]&&\\
000 \ar[r]&200\ar[ru] \ar[rd]&\oplus &010\ar[ru]\ar[rd]&\oplus&200\ar[r]&000\\
&&100\ar[rd] \ar[ru]&\oplus&100\ar[ru]&&\\
&&&200\ar[ru]\ar[ruuu]&&&\\}
$$
where each triple of the numbers determines a weight a $Sp(6,\mathbb{R})$--submodule that is irreducible component of $\mathcal{H}^i(L^i(T))$ and the arrows characterize the possible (from representation viewpoint) components of the BGG--like operators between $\mathcal{H}^i(L^i(T))$ and $\mathcal{H}^{i+1}(L^i(T))$.

The tractor--like connection has formula
$$
\nabla^T_a\begin{pmatrix}
r\\
s_d\\
t\\
\end{pmatrix}=
\begin{pmatrix}
\nabla^0_ar\\
\nabla^0_as_d+J_{ad}r\\
\nabla^0_at+s_a\\
\end{pmatrix}$$
and the twisted and covariant exterior differentials coincide and have the formula
\begin{align*}
d^T\begin{pmatrix}
r\\
s_d\\
t\\
\end{pmatrix}_{a_1\dots a_k}&=r_0\circ (\sum_i (-1)^{i+1}
\nabla_{a_i}^T\begin{pmatrix}
r_{a_1\dots a_k}\\
s_{a_1\dots a_kd}\\
t_{a_1\dots a_k}\\
\end{pmatrix})\\
&=
\begin{pmatrix}
\sum_i (-1)^{i+1}\nabla^0_{a_i}r_{a_1\dots a_k}\\
\sum_i (-1)^{i+1}\nabla^0_{a_i}s_{a_1\dots a_kd}+\sum_i(-1)^{i+1}J_{a_id}r_{a_1\dots a_k}\\
r_0(\sum_i (-1)^{i+1}\nabla^0_{a_i}t_{a_1\dots a_k}+\sum_i (-1)^{i+1}s_{a_1\dots a_ka_i})\\
\end{pmatrix},
\end{align*}
where $r_0$ on the third slot is projection onto the totally trace--free part.

The composition $\partial^*d^T$ is given by the formula
\begin{align*}
\partial^*d^T\begin{pmatrix}
r\\
s_d\\
t\\
\end{pmatrix}_{a_1\dots a_k}&=
(-1)^{k+1}(k+1)[\begin{pmatrix}
(n-k)(-1)^kr_{a_1\dots a_{k}}\\
r_0(\sum_i (-1)^{i+1}s_{a_1\dots a_kda_i}+(-1)^ks_{a_1\dots a_kd})\\
0\\
\end{pmatrix}\\
&+\begin{pmatrix}
\sum_i (-1)^{i+1}\nabla^0_{a_i}s_{a_1\dots a_kd}{}^{d}+(-1)^k\nabla^0_{d}s_{a_1\dots a_k}{}^{d}\\
r_0(\sum_i (-1)^{i+1}\nabla^0_{a_i}t_{a_1\dots a_kd}+(-1)^k\nabla^0_{d}t_{a_1\dots a_k})\\
0\\
\end{pmatrix},
\end{align*}
where $r_0$ is again projection onto the totally trace--free part. We observe that the eigenvalues of $\partial^*d^T$ on irreducible components of $Im(\partial^*)$ are $-(k+1)(n-k)$ for $0\leq k<n$ on the top slot and $-(k+1)$ for $0\leq k <n/2$ on image of $r_0$ in the middle slot. This allows us to construct all splitting operators $\mathcal{L}_i(d^T)$ and the corresponding BGG--like operators. As an example, we write down only the precise formulas for the first two splitting $\mathcal{L}_i(d^T)$ and BGG--like operators $\mathcal{B}_i$ in the case $n=6$:
\begin{align*}
\mathcal{L}_0(\nabla^T)\begin{pmatrix}
0\\
0\\
t\\
\end{pmatrix}&=
\begin{pmatrix}
0\\
0\\
t\\
\end{pmatrix}-\begin{pmatrix}
0\\
\nabla^0_{d}t\\
0\\
\end{pmatrix}
-\frac16\begin{pmatrix}
\nabla^0_{d}(\nabla^0)^{d}t\\
0\\
0\\
\end{pmatrix}=
\begin{pmatrix}
-\frac16\nabla^0_{i}(\nabla^0)^{i}t\\
-\nabla^0_{d}t\\
t\\
\end{pmatrix},\\
\mathcal{B}_0(t)&=\pi\circ \nabla_{a_1}^T\circ \begin{pmatrix}
-\frac16\nabla^0_{i}(\nabla^0)^{i}t\\
-\nabla^0_{d}t\\
t\\
\end{pmatrix}=
\pi\circ \begin{pmatrix}
*\\
-\nabla^0_{a_1}\nabla^0_{d}t-J_{a_1d}\frac16\nabla^0_{i}(\nabla^0)^{i}t\\\
0\\
\end{pmatrix}\\
&=-\nabla^0_{(a_1}\nabla^0_{d)}t,\\
\mathcal{L}_1(d^T)\begin{pmatrix}
0\\
s_{(a_1d)}\\
0\\
\end{pmatrix}&=
\begin{pmatrix}
0\\
s_{(a_1d)}\\
0\\
\end{pmatrix}-\frac1{10}\begin{pmatrix}
\nabla^0_{d}(s_{a_1}{}^{d}+s^{d}{}_{a_1})\\
0\\
0\\
\end{pmatrix}=\begin{pmatrix}
-\frac1{10}\nabla^0_{d}(s_{a_1}{}^{d}+s^{d}{}_{a_1})\\
s_{(a_1d)}\\
0\\
\end{pmatrix},\\
\mathcal{B}_1(s_{(a_1d)})&=\nabla^0_{a_1}s_{(a_2d)}-\nabla^0_{a_2}s_{(a_1d)}-J_{a_1d}\frac1{10}\nabla^0_{i}(s_{a_2}{}^{i}+s^{i}{}_{a_2})\\
&+J_{a_2d}\frac1{10}\nabla^0_{i}(s_{a_1}{}^{i}+s^{i}{}_{a_1}).
\end{align*}

If we are on a CF--structure, i.e., $H=S=0$, then using the Proposition \ref{prop6.2} and Theorem \ref{escur} we obtain that
\begin{align*}
(d^T)^2\begin{pmatrix}
r\\
s_d\\
t\\
\end{pmatrix}_{a_1\dots a_k}&=
\begin{pmatrix}
0\\
\sum_{i<j} (-1)^{i+j+1}(R^0)_{a_ia_jd}{}^{e}s_{a_1\dots a_ke}\\
0\\
\end{pmatrix},
\end{align*}
which vanishes if and only if $R^0=0$. In such case, the BGG--like sequence is a complex.

\end{document}